\documentclass{amsart} 

\usepackage{mathtools}
\usepackage{bbm}
\usepackage{cjhebrew}
\usepackage{amsmath, amssymb}
\usepackage{amsthm} 
\usepackage{amsfonts}
\usepackage{amscd}\usepackage{amsxtra}
\usepackage[all,cmtip]{xy}
\newcommand{\sG}{\mathcal{G}}
\newcommand{\sH}{\mathcal{H}}
\newcommand{\sI}{\mathcal{I}}

\newcommand{\sL}{\mathcal{L}}
\newcommand{\sM}{\mathcal{M}}
\newcommand{\sO}{\mathcal{O}}

\newcommand{\sX}{\mathcal{X}}
\newcommand{\sY}{\mathcal{Y}}

\newcommand{\bX}{\mathbb{X}}

\newcommand{\bL}{\mathbb{L}}

\newcommand{\bN}{\mathbb{N}}
\newcommand{\bZ}{\mathbb{Z}}

\newcommand{\bA}{\mathbb{A}}
\newcommand{\bB}{\mathbb{B}}
\newcommand{\bP}{\mathbb{P}}

\newcommand{\fm}{\mathfrak{m}}
\newcommand{\fn}{\mathfrak{n}}
\newcommand{\fC}{\mathfrak{C}}

\newcommand{\fx}{\mathfrak{x}}

\newcommand{\fv}{\mathfrak{v}}

\newcommand{\fl}{\mathfrak{l}}

\newcommand{\fs}{\mathfrak{s}}

\newcommand{\Spec}{\operatorname{Spec}}
\newcommand{\Ker}{\operatorname{Ker}}

\newcommand{\ord}{{\operatorname{ord}}}

\newcommand{\Mor}{\operatorname{Mor}}

\renewcommand{\bB}{\mathbb{B}}

\newcommand{\Arc}{\operatorname{\bf Arc}}

\newcommand{\Form}{\operatorname{\bf Form}}

\newcommand{\op}{{\text{\rm op}}}

\newcommand{\by}[1]{\overset{#1}{\longrightarrow}}

\newcommand{\iso}{\by{\sim}}

\newcommand{\inj}{\hookrightarrow}

\newcommand{\surj}{\rightarrow\!\!\!\!\!\rightarrow}
 
\newcommand{\colim}{\varinjlim}
\renewcommand{\lim}{\varprojlim}

\renewcommand{\op}{\operatorname}

\renewcommand \dim[1]{\mbox{dim}{#1}}

\newcommand\sch[1]{\categ{Sch}_{#1}}
\newcommand\set{\mathbf{Set}}
\newcommand\im[1]{\operatorname{im}({#1})}

\newcommand\spec[1]{\operatorname{Spec}(#1)}

\newcommand\fatpoints[1]{\categ {Fat}_{#1}}

\newcommand\categ[1]{\mathbbmss{#1}}

\newcommand\into{\hookrightarrow}

\newcommand\sieves[1]{\categ{Sieve}_{#1}}
\newcommand\fim[1]{\texttt{Im}(#1)}

\newcommand \lef{{\mathbb L}}
\newcommand\func[1]{#1^\circ}
\newcommand \grot[1]{{\mathbf {Gr}(#1)}}

\newcommand \nat{\mathbb N}

\newcommand \onto{\twoheadrightarrow}

\newcommand \mor[3]{\op{Mor}_{#1}(#2,#3)}
\newcommand \zet{\mathbb Z}

\newcommand \complet[1]{\widehat {#1}}

\catcode`\@=11
\def\opn#1#2{\def#1{\mathop{\kern0pt\fam0#2}\nolimits}} 
\def\underrightarrow{\mathpalette\underrightarrow@}
\def\underrightarrow@#1#2{\vtop{\ialign{$##$\cr
 \hfil#1#2\hfil\cr\noalign{\nointerlineskip}%
 #1{-}\mkern-6mu\cleaders\hbox{$#1\mkern-2mu{-}\mkern-2mu$}\hfill
 \mkern-6mu{\to}\cr}}} 

\def\underleftarrow{\mathpalette\underlefalwaystarrow@}
\def\underleftarrow@#1#2{\vtop{\ialign{$##$\cr
 \hfil#1#2\hfil\cr\noalign{\nointerlineskip}#1{\leftarrow}\mkern-6mu
 \cleaders\hbox{$#1\mkern-2mu{-}\mkern-2mu$}\hfill
 \mkern-6mu{-}\cr}}}
    \let\phi=\varphi
    \let\epsilon=\varepsilon  
\def\:{\colon}   
\let\oldtilde=\tilde
\def\tilde#1{\mathchoice{\widetilde{#1}}{\widetilde{#1}}%
{\indextil{#1}}{\oldtilde{#1}}}
\def\indextil#1{\lower2pt\hbox{$\textstyle{\oldtilde{\raise2pt%
\hbox{$\scriptstyle{#1}$}}}$}}
\def\pnt{{\raise1.1pt\hbox{$\textstyle.$}}}  
\let\amp@rs@nd@\relax
\newdimen\ex@
\newdimen\bigaw@
\newdimen\minaw@
\newdimen\minCDaw@  
\newif\ifCD@
\def\minCDarrowwidth#1{\minCDaw@#1}
\renewenvironment{CD}{\@CD}{\@endCD}
\def\@CD{\def\A##1A##2A{\llap{$\vcenter{\hbox
 {$\scriptstyle##1$}}$}\Big\uparrow\rlap%
{$\vcenter{\hbox{$\scriptstyle##2$}}$}&&}%
\def\V##1V##2V{\llap{$\vcenter{\hbox
 {$\scriptstyle##1$}}$}\Big\downarrow\rlap%
{$\vcenter{\hbox{$\scriptstyle##2$}}$}&&}%
\def\={&\hskip.5em\mathrel
 {\vbox{\hrule width\minCDaw@\vskip3\ex@\hrule width
 \minCDaw@}}\hskip.5em&}%
\def\verteq{\Big\Vert&&}%
\def\noarr{&&}%
\def\vspace##1{\noalign{\vskip##1\relax}}%
\relax\let\amp@rs@nd@&\iffalse}\fi
 \CD@true\vcenter\bgroup\relax%
\let\\=\cr\iffalse}\fi\tabskip\z@skip\baselineskip20\ex@
 &\hfill$\m@th##$\hfill\cr}
\def\@endCD{\cr\egroup\egroup}
\def\>#1>#2>{\amp@rs@nd@\setbox\z@\hbox{$\scriptstyle
 \;{#1}\;\;$}\setbox\@ne\hbox{$\scriptstyle\;{#2}\;\;$}\setbox\tw@
 \hbox{$#2$}\ifCD@
 \global\bigaw@\minCDaw@\else\global\bigaw@\minaw@\fi
 \ifdim\wd\z@>\bigaw@\global\bigaw@\wd\z@\fi
 \ifdim\wd\@ne>\bigaw@\global\bigaw@\wd\@ne\fi
 \ifCD@\hskip.5em\fi
 \ifdim\wd\tw@>\z@
 \mathrel{\mathop{\hbox to\bigaw@{\rightarrowfill}}\limits^{#1}_{#2}}\else
 \mathrel{\mathop{\hbox to\bigaw@{\rightarrowfill}}\limits^{#1}}\fi
 \ifCD@\hskip.5em\fi\amp@rs@nd@}
\def\<#1<#2<{\amp@rs@nd@\setbox\z@\hbox{$\scriptstyle
 \;\;{#1}\;$}\setbox\@ne\hbox{$\scriptstyle\;\;{#2}\;$}\setbox\tw@
 \hbox{$#2$}\ifCD@
 \global\bigaw@\minCDaw@\else\global\bigaw@\minaw@\fi
 \ifdim\wd\z@>\bigaw@\global\bigaw@\wd\z@\fi
 \ifdim\wd\@ne>\bigaw@\global\bigaw@\wd\@ne\fi
 \ifCD@\hskip.5em\fi
 \ifdim\wd\tw@>\z@
 \mathrel{\mathop{\hbox to\bigaw@{\leftarrowfill}}\limits^{#1}_{#2}}\else
 \mathrel{\mathop{\hbox to\bigaw@{\leftarrowfill}}\limits^{#1}}\fi
 \ifCD@\hskip.5em\fi\amp@rs@nd@}
\def\@CDS{\def\A##1A##2A{\llap{$\vcenter{\hbox
 {$\scriptstyle##1$}}$}\Big\uparrow\rlap%
{$\vcenter{\hbox{$\scriptstyle##2$}}$}&}%
\def\V##1V##2V{\llap{$\vcenter{\hbox
 {$\scriptstyle##1$}}$}\Big\downarrow\rlap%
{$\vcenter{\hbox{$\scriptstyle##2$}}$}&}%
\def\={&\hskip.5em\mathrel
 {\vbox{\hrule width\minCDaw@\vskip3\ex@\hrule width
 \minCDaw@}}\hskip.5em&}
\def\verteq{\Big\Vert&}
\def\novarr{&}
\def\noharr{&&}
\def\SE##1E##2E{\slantedarrow(0,18)(4,-3){##1}{##2}&}
\def\SW##1W##2W{\slantedarrow(24,18)(-4,-3){##1}{##2}&}
\def\NE##1E##2E{\slantedarrow(0,0)(4,3){##1}{##2}&}
\def\NW##1W##2W{\slantedarrow(24,0)(-4,3){##1}{##2}&}
\def\slantedarrow(##1)(##2)##3##4{\thinlines\unitlength1pt%
\lower 6.5pt\hbox{\begin{picture}(24,18)%
\put(##1){\vector(##2){24}}%
\put(0,8){$\scriptstyle##3$}%
\put(20,8){$\scriptstyle##4$}%
\end{picture}}}
\def\vspace##1{\noalign{\vskip##1\relax}}\relax%
\let\amp@rs@nd@&\iffalse}\fi
 \CD@true\vcenter\bgroup\relax\let\\=\cr\iffalse}\fi\tabskip\z@skip\baselineskip20\ex@
 &\hfill$\m@th##$\hfill\cr}
\def\@endCDS{\cr\egroup\egroup}
\theoremstyle{plain} 
\newtheorem{theorem}{\indent\sc Theorem}[section]
\newtheorem{lemma}[theorem]{\indent\sc Lemma}
\newtheorem{corollary}[theorem]{\indent\sc Corollary}
\newtheorem{proposition}[theorem]{\indent\sc Proposition}

\newtheorem{conjecture}[theorem]{\indent\sc Conjecture}
\theoremstyle{definition} 
\newtheorem{definition}[theorem]{\indent\sc Definition}
\newtheorem{remark}[theorem]{\indent\sc Remark}
\newtheorem{example}[theorem]{\indent\sc Example}

\newtheorem{question}[theorem]{\indent\sc Question}
\makeatletter
\def\address#1#2{\begingroup
\noindent\parbox[t]{7.8cm}{%
\small{\scshape\ignorespaces#1}\par\vskip1ex
\noindent\small{\itshape astout@gc.cuny.edu}%
\/: #2\par\vskip4ex}\hfill%
\endgroup}%
\makeatother
\begin{document}

\title{Arc stability and schemic motivic integration}

\author{Andrew R. Stout}

\email{astout@gc.cuny.edu}

\maketitle

\begin{abstract}
We develop the notion of a functorial motivic volume relative to any complete noetherian local ring, which takes into account the nilpotent structure of a scheme of finite type over a field. We show how the schemic motivic measure encodes data from the local deformations of a scheme. We offer some preliminary applications such as rationality results on the  motivic poincar\'{e} series. We postulate that there is an unknown algebraic condition on the theory of local artinian rings which will force the reduction of the auto-arcs to be affine $n$-space. This is a foundational open problem for schemic motivic integration.
\end{abstract}
\setcounter{tocdepth}{1} \tableofcontents{}

\section*{Introduction} 
In \cite{Sch1} and \cite{Sch2}, Schoutens introduced the possibility of developing a topos-theoretic analogue of motivic integrals. However, the integrals defined therein are finite in nature. The main thesis of this paper is that one can naturally extend the notion of geometric motivic integration to the setting of schemic grothendieck rings. Moreover, it  is shown how to reduce these schemic integrals to the usual notion of an integral in geometric motivic integration.

In Section 1, we introduce some background for understanding the topos theoretic approach. In Section 2, we introduce some notions of stability for schemes which are important for evaluating a schemic integral, and we show how stability for a scheme naturally gives rise to a measure on the generalizd arc space. What is of note here is that the integals which take place with respect to this measure take into account the nilpotent structure of the underlying scheme. 
In Section 3, we push the material further by showing how stability is connected with ideas coming from deformation theory. Furthermore, in this section, we make it clear how to relate the schemic integral to the geometric motivic integral. Finally, in this section, we also discuss a couple of applications.

\section{Background}\label{finite}

We now quickly give an introduction to Schoutens' theory of finite schemic integration. Much of what is stated here is taken directly from \cite{Sch1} and \cite{Sch2}. Let $\sch\kappa$ be the category of separated schemes of finite type over a field $\kappa$. We form the Grothendieck topology $\categ{Form}_{\kappa}$ on $ \sch\kappa$ in the following way.

\begin{definition}\label{morofsieves}
Given two sieves $\sX$ and $\sY$, we say
that a natural transformation $\nu : \sY \to \sX$ is {\it a morphism of
 sieves} if
given any morphism of schemes $\varphi : Z \to Y$ such that
$\im{\func{\varphi}} \subset \sY$,
there exists a morphism of schemes $\psi : Z \to X$ with $\sX
\subset X$ such that the following diagram commutes 
\[\xymatrix{\entrymodifiers={+!!<0pt,\fontdimen22\textfont2>}&Z^{\circ}  \ar[d]_{\func{\varphi}}  \ar[rrd]^{\func{\psi}}\\&\sY \ar[r]_{\nu} &\sX \ar[r]_{\iota}&X^{\circ} } \]
where $\iota$ is the natural inclusion defining $\sX$ as a subfunctor of $X^{\circ} := \mbox{hom}_{\sch\kappa}(-,X)$.
This forms a category which we again denote by $\sieves\kappa$.
\end{definition}

We say that a sieve $\sY$ is {\it subschemic} if it is of the form $\fim{\varphi^{\circ}}$ where $\varphi : X \to Y$ is a morphism in $\sch{\kappa}$. The collection of subschemic sieves satisfies the axioms of a Grothendieck pre-topology; however, it is not all that interesting as we have the following theorem due to Schoutens:

\begin{theorem}
Let $\nu : \sY \to \sX$ be a morphism in $\sieves\kappa$ and assume that $\sX$ and $\sY$ are subschemic. Then, $\nu$ is rational -- i.e., there exists a morphism $\varphi:Y \to X$ in $\sch{\kappa}$ such that 
\begin{equation}
\varphi^{\circ} \circ\iota = \nu \ ,
 \end{equation}
where $\iota: \sY \into Y^{\circ}$ is a natural inclusion.
\end{theorem}

However, there is a large class of sieves which do not have this property. Recall the construction of a formal scheme. One starts with a closed subscheme $Y$ of $X$ with corresponding ideal sheaf $\sI_Y$. For each $n \in \nat$, $\sI_{Y}^{n}$ is a quasi-coherent sheaf of ideals of $\sO_{X}$. Thus, we have the closed subscheme $Y_n$ of $X$ determined by the ideal sheaf $\sI_{Y}^{n}$. Then, the formal scheme of $X$ along $Y$ is the  locally ringed topological space $\complet Y$ which is isomorphic to $\colim_{n\in\nat} Y_n$. This leads us to make the following definition.
\begin{definition}
We say that a sieve $\sX$ is {\t formal} if for each connected finite ${\kappa}$-scheme $\fm$, there is a subschemic sieve $\sY_{\fm}\subset\sX$ such that the sets
$\sY_{\fm}(\fm)$ and $\sX(\fm)$ are equal.
\end{definition}
In Theorem $7.8$ of \cite{Sch1}, Schoutens proved that the collection of all formal sieves, denoted by $\categ{Form}_{\kappa}$ is a Grothendieck pre-topology. It can be shown as well that categorical product and coproduct commute in the full subcategory  $\categ{Form}_{\kappa}$ of $\sieves\kappa.$ Thus, we may form the Grothendieck ring of  $\categ{Form}_{\kappa}$ . We denote the resulting ring by  $\grot{\categ{Form}_{\kappa}}$ and call it {\it the Grothendieck ring of the formal motivic site}. There is a surjective ring homomorphism
\begin{equation}
\grot{\categ{Form}_{\kappa}} \onto \grot {\mathbf{Var}_{\kappa}} \ .
\end{equation}

In motivic integration, one often deals with the arc space $\mathcal{L}(X)$ which is the projective limit of the  $n$-th order arc spaces. The truncated arc space $\sL_{n}(X)$ is defined to be the separated scheme of finite type over $\kappa$ representing the functor from connected $\kappa$-schemes which are finite over $\kappa$ to $\set$:
\begin{equation}
\fm \mapsto X^{\circ}(\fm \times_{\kappa}\spec{\kappa[t]/(t^{n})}) \ .
\end{equation}
Usually, one only considers the reduced structure on $\sL(X)$.

Let $\fatpoints\kappa$ be the full subcategory of $\sch\kappa$ whose objects are connected finite $\kappa$-schemes. We call $\fm \in\fatpoints\kappa$ a {\it fat point} over $\kappa$. All sieves $\sX$ restrict to $\fatpoints\kappa$ . We will abuse notation and denote the restriction of a sieve $\sX$ to $\fatpoints\kappa$ as $\sX$ as well. Moreover, we will denote the resulting category of all sieves $\sX$ restricted to $\fatpoints\kappa$ by $\sieves\kappa$.
The reason that we may perform this restriction is due to the following fact.

\begin{theorem}
Let $X$ and $Y$ be closed subschemes contained in a separated $\kappa$-scheme $Z$ of finite type. Then, $X$ and $Y$ are non-isomorphic over $\kappa$ if and only if there exists $\fm\in\fatpoints\kappa$ such that  $X^{\circ}(\fm)$ and $Y^{\circ}(\fm)$
are distinct subsets of $Z^{\circ}(\fm)$.
\end{theorem}

\begin{proof}
This is a restatement of Lemma 2.2 of \cite{Sch1}. A proof can be found there.
\end{proof}

One of the insights of Schoutens was that the construction of the arc space works just as well when we replace $\spec{\kappa[t]/(t^{n})}$ with an arbitrary fat point $\fn$. This leads us to define the generalized arc space of a sieve $\sX$ along the fat point $\fn$ by
\begin{equation}\label{arcop}
\nabla_{\fn}\sX(-) := \sX(-\times_{\kappa}\fn)
\end{equation}
as a functor from $\fatpoints\kappa$ to $\set$. Schoutens proved in \S $3$ of \cite{Sch2} that if $\sX = X^{\circ}$ for some $X\in\sch\kappa$, then it follows that $\nabla_{\fn}\sX$ is a represented by an element of $\sch\kappa$. Thus, it follows immediately that $\nabla_{\fn}\sX \in \sieves\kappa$ for any $\sX \in \sieves\kappa$ and any $\fn\in\fatpoints\kappa$. Moreover, Schoutens showed that if $\sX$ is formal, then so is $\nabla_{\fn}\sX$. It is natural to define the (weightless) finite schemic measure to be
\begin{equation}
\mu_{\fn}(\sX) := [\nabla_{\fn}\sX]\lef^{-\dim{\nabla_{\fn}\sX}} \ 
\end{equation}
in the Grothendieck ring $\grot{\categ{Form}_{\kappa}}_{\lef}$ where $\lef = [\bA_{\kappa}^{1}]$.
More or less, the integrals which take place in \cite{Sch2} are finite sums of elements of $\grot{\categ{Form}_{\kappa}}_{\lef}$ of the form $\mu_{\fn}(\sX)$.

\section{Geometric  integration for schemes} 

In this section, we develop the theory of geometric schemic integration over certain infinite arcs and prove a weak change of variables formula for this theory. To prove this change of variables formula, we display a way to specialize this type of integration to the theory of geometric motivic integration for varieties as developed in \cite{DL1}. 

\subsection{Admissible arcs.}

Geometric motivic integration takes place over $\spec{\kappa[[t]]}$. Moreover, as we saw in the introduction, we would like to make use of the generalized arc operator. Thus, our first step in constructing a schemic integral is to form certain colimits of fat points, which we will term {\it arcs}. As a special case, we will recover the arc $\spec{\kappa[[t]]}$ used in the classical theory.
Fix a field $\kappa$.
 Let $R$ be a complete noetherian local ring with maximal ideal $\sM$. We define
the following:

\begin{itemize}\label{bullets}
\item $R_n : = R/\sM^n$ for $n\in\bN\setminus\{0\}$.
\item $\fs_n := (\Spec R_n, \sO_{\Spec R_n})$.
\item Let $\fs_{n-1} \inj \fs_{n}$ be the closed immersion defined by the
surjective ring homomorphism $R_{n} \surj R_{n-1}$.
\item Let $\bX$ denote the resulting directed system of schemes and assume that $\bX$ is not a finite set.
\item Let $\fx = \colim\bX$  denote the direct limit of this directed system in
the category of locally ringed spaces.
\item Note that $\fx$ is the local ringed space $(\{x\} , \sO_{\mathfrak{x}})$
where $\sO_{\mathfrak{x}}(\{x\}) = R$. 
\item ({\it Working over the field} $\kappa$) We will make the assumption that
$$\fs_n \in \sch{\kappa}$$ where $\sch{\kappa}$ is the category of separated schemes of finite type
over the field $\kappa$.
\item We denote by $\mathbf{Var}_{\kappa}$ the full sub-category of $\sch{\kappa}$ whose objects are
objects $X$ of $\sch{\kappa}$ such that $X = X^{red}$  where $X^{red}$ is the
reduction of $X$. We will call an object in $\mathbf{Var}_{\kappa}$ a {\it variety}.
\end{itemize}

 The category $\fatpoints{\kappa}$ of {\it fat points over the field } ${\kappa}$ is defined as the full sub-category of 
the category of $\sch{\kappa}$ formed by all $\fm \in \sch{\kappa}$ such that $\fm^{red} = \Spec {\kappa} $. 

\begin{theorem} Let $\fm \in \sch{\kappa}$. When $\kappa$ is algebraically closed, the following are equivalent.

\begin{enumerate}
\item $ \fm \in \fatpoints{\kappa}$.
\item $\fm$ is the spectrum of a local artinian ring with residue 
field ${\kappa}$.
 \item The underlying topological space of $\fm$ is homeomorphic to the one point topological space.
 \item $\fm$ is connected and of dimension $0$.
 \item $\fm$ is isomorphic to $J_{o}^{n} X$ for some closed subscheme $X = \Spec A$ of 
	$\bA_{{\kappa}}^{g}$, where $o \in X$ is a closed point corresponding to a maximal ideal $\sM$ of $A$ and 
	$J_{o}^{n} X := \Spec(A/\sM^{n})$ . 
 \end{enumerate}
\end{theorem}

We let $\Arc_{\kappa}$ denote the full sub-category of locally ringed spaces whose objects are locally ringed spaces $\fx$ constructed precisely as in \ref{bullets}. We call
it the {\it category of admissible arcs over} ${\kappa}$.  Sometimes we will also write the objects of $\Arc_{\kappa}$ as  $(\fx, \bX)$ when $\fx$ and $\bX$ are as in \ref{bullets}. There is no danger here as $\fx$ uniquely determines $\bX$ and vise verse. This is just a notational convience.

\begin{example} Let $\fl_n := \Spec {\kappa}[x]/(x^n)$. Then,
\begin{equation}
\fl := \colim\{\fl_n\}=\colim_n \fl_n =  \Spec {\kappa}[[x]] \in\Arc_{\kappa} \ .
\end{equation}
\end{example}

\begin{theorem} \label{arcschar}  Every element
	 $(\fx,\bX)$ of  $\Arc_{\kappa}$ is isomorphic to $(J_{o}^{\infty} X, \{J_{o}^{n}X\})$ for some closed subscheme $X$ of 
	$\bA_{{\kappa}}^{g}$ where $o \in X$ is a closed point.
\end{theorem}

Here $J_{o}^{n}X$ denotes the subscheme of $X$ determined by the $n$th power of the maximal ideal of $o$ and $J_{o}^{\infty} X$ denotes the filtered colimit in locally ringed spaces (i.e., it is a formal scheme).  Therefore, an immediate corollary of this theorem is that if $(\fx , \bX)\Arc_{\kappa}$, then  $\bX$ is a collection of fat points such that the embedded dimension of
	every $\fm \in \bX$ is bounded by some natural number $g$.
We have a set map $\ell : \fatpoints{\kappa} \to \bN$
defined by setting $\ell(\fm)$ to be the dimension of $\sO_{\fm}(\fm)$ as
a vector space over ${\kappa}$ -- i.e., the cardinality of the basis elements used to generate the global sections of $\fm$ as 
a vector space over $\kappa$. We call $\ell(\fm)$ the {\it length of} $\fm$. 

\subsection{Notions of stability for schemes.}\label{1}

Let $\fx \in \Arc_{\kappa}$ and $X\in\sch{\kappa}$, we define $\nabla_{\fx} X$ to be the projective limit in the category of locally ringed spaces of the indirected set of schemes $\{\nabla_{\fn}X\mid \fn \in \bX\}$. For a definition of $\nabla_{\fn}X$, the reader may consult Equation \ref{arcop} of \S \ref{finite}. The reader may also consult \S 3 and \S 4 of \cite{Sch2} for more details. The important fact here is that any morphism $\fn \to \fm$ of fat points induces a natural transformation of functors $\nabla_{\fm} \to \nabla_{\fn}$ so that the definition of $\nabla_{\fx}X$ makes sense. 

\begin{definition} Let $\fx \in \Arc_{\kappa}$.
 We will say that $X\in\sch{\kappa}$ is $\fx$-\textit{quasi-stable} if there exists a
positive 
integer $N$ such that for all $\fn \in \bX$  with $\ell(\fn) \geq N$, the
functor  $$(\nabla_{\fx/\fn}X)^{\circ} : \fatpoints{\kappa} \to \set$$ is a formal sieve. 
\end{definition}

In the above definition, $\nabla_{\fx/\fn}X$ denotes the image of $\nabla_{\fx}X$ in $\nabla_{\fn}X$
under the natural map $\pi_{\fn}^{\fx}$.  Here, as usual, $Y^{\circ}$ is the functor
$\fatpoints{\kappa} \to \mbox{Sets}$ defined by $Y^{\circ}(\fm) = \mor{\sch{\kappa}}{\fm}{Y}$ whenever $Y\in \sch\kappa$.
By Theorem 8.1 of \cite{Sch2}, 
every $X \in \sch{\kappa}$ is $\fx$-quasi-stable where $\fx = J_{p}^{\infty} X$ and $p$ is some closed point of $X$. Thus, by Theorem 8.1 of \cite{Sch2}, we have  the following proposition.

\begin{proposition}\label{quasistableprop}
 For every $X \in \sch{\kappa}$, $X$ is $\fx$-quasi-stable for all $\fx\in\Arc_{\kappa}$. 
\end{proposition}

We denote by $qs_{\fx}(X)$ the minimum
positive integer such that 
$$ \exists \fn \in \bX, \ qs_{\fx}(X) = \ell(\fn) \ \ \mbox{and} \ \
(\nabla_{\fx/\fn}X)^{\circ}\in\categ{Form}_{\kappa} \ .$$
In this case, we say that $X$ is $\fx$-\textit{quasi-stable at level}
$qs_{\fx}(X)-1$. We call it the {\it quasi-stability function}.  A direct extension of the previous proposition is the following
\begin{proposition} It is the case that
$qs_{\fx}(X) = 1$ for all $\fx \in \Arc_{\kappa}$ and all $X\in \sch{\kappa}$. Thus, every element $X$ of $\sch{\kappa}$ is $\fx$-quasi-stable at level $0$ for all $\fx\in\Arc_{\kappa}.$
\end{proposition}
The following is a list progressively stronger properties built out of the
notion of quasi-stability: 

\begin{itemize}
\item ({\it Weak Stability}). Suppose that $X$ is $\fx$-quasi-stable. This
implies that there exists fat points  $\fm \geq
\fn$ in $\bX$ such that
$$(\nabla_{\fx/\fm}X)^{\circ} \ \ \ \ \mbox{and} \ \ \ \
(\nabla_{\fx/\fn}X)^{\circ}$$ are formal sieves whenever $\ell(\fn)>>0$. If it
is the case that $\nabla_{\fx/\fm}X$ is the union of fibers of the natural map
$$\pi_{\fn}^{\fm} : \nabla_{\fm}X \to \nabla_{\fn}X$$ when its range is
restricted to $\nabla_{\fx/\fn}X$ and when the length $\ell(n)$ above is
sufficiently large, then we say that $X$ is $\fx$-{\it weakly stable}. 
\item ({\it Lax Stability}) Assume that $X$ is $\fx$-weakly stable and that the
induced map $$\nabla_{\fx/\fm}X \to \nabla_{\fx/\fn}X$$
given above is a piecewise trivial fibration with fiber $\bA_{{\kappa}}^{r}$ when
$\ell(n) >>0$. In this case, we say that $X$ is $\fx$-{\it laxly stable}.
\item ({\it Stability}) Suppose that $X$ is $\fx$-laxly stable so that, in the
above, $r = d(\ell(\fm) - \ell(\fn))$ where $d=\dim X$. Then, we say that $X$ is
$\fx$-{\it stable}
\end{itemize}

Each of these conditions has a corresponding function in analogy to the
quasi-stability function $qs$. We will illustrate how this works by defining the
stability function:
When $X$ is $\fx$-stable, we will denote by $s_{\fx}(X)$ the minimum positive
integer such that $\forall \fm, \fn \in \bX$ with $\ell(\fm) > \ell(\fn) \geq
s_{\fx}(X)$, it is the case that
$$(\nabla_{\fx/\fm}X)^{\circ} \ \ \ \ \mbox{and} \ \ \ \
(\nabla_{\fx/\fn}X)^{\circ}$$ are formal sieves and that the natural map
$$\pi_{\fn}^{\fm} : \nabla_{\fm}X \to \nabla_{\fn}X \ ,$$ when its range is
restricted to $\nabla_{\fx/\fn}X$, is a piecewise trivial fibration with fiber
$\bA_{{\kappa}}^{r}$ where $r = d(\ell(\fm) - \ell(\fn))$ and $d = \dim X$. 
In this case, we say that $X$ is $\fx$-\textit{stable at level} $s_{\fx}(X)-1$.
when $X$ is not $\fx$-stable, we set $s_{\fx}(X) = +\infty$. Note that $s$ is a
function $$s : \sch{\kappa} \times \Arc_{\kappa} \to 
\bN\cup\{+\infty\}$$ which we call it the \textit{stability function}. We leave it
to the reader to define the {\it weak stability function} $ws$ and the {\it  lax
stability function} $ls$ as they are defined in exactly the same way.

For each $\fx \in \Arc_{\kappa}$, we define the following subsets of $\sch{\kappa}$:
\begin{itemize}
\item The collection of all $\fx$-weakly-stable separated ${\kappa}$-schemes of finite type: 
$$WStS_{\fx}:= ws_{\fx}^{-1}(\bN) \ .$$
\item The collection of all $\fx$-laxly-stable separated ${\kappa}$-schemes of finite type:
$$LStS_{\fx}:= ls_{\fx}^{-1}(\bN) \ .$$
\item The collection of all $\fx$-stable separated ${\kappa}$-schemes of finite type: $$StS_{\fx}
:= s_{\fx}^{-1}(\bN) \ .$$
\end{itemize}

\begin{theorem} \label{smoothstabletheorem}
Let $\categ{SmSch}_{\kappa}$ be the full subcategory of $\sch{\kappa}$ formed by 
separated smooth schemes of finite type over the field ${\kappa}$. 
 For all $\fx \in \Arc_{\kappa}$ and all $X \in \categ{SmSch}_{\kappa}$, $X$ is $\fx$-stable at level $0$. Hence,
 for all $\fx \in \Arc_{\kappa}$, and all $X\in \categ{SmSch}_{\kappa}$
 \begin{equation*}
 s_{\fx}(X) = ls_{\fx}(X) = ws_{\fx}(X) = 1 \ .
 \end{equation*}
Thus, by regarding $\categ{SmSch}_{\kappa}$ as a set, we have $\categ{SmSch}_{\kappa} \subset StS_{\fx}$ for all $ \fx \in \Arc_{\kappa}$. 
\end{theorem}

\begin{proof}
 This is proved in Theorem 4.14 of \cite{Sch2}. 
\end{proof}

\noindent Clearly, for all $\fx \in \Arc_{\kappa}$,
\begin{equation}
\categ{SmSch}_{\kappa}\subset StS_{\fx} \subset LStS_{\fx} \subset
WStS_{\fx} \subset \sch{\kappa}  \ .
\end{equation}

\begin{conjecture} \label{stableconj} For each $\fx \in \Arc_{\kappa}$ of positive dimension, each inclusion above is strict.
\end{conjecture}

\begin{remark}
Note that by design all of these notions of stability are vacuous if we the arc $\fx$ has dimension zero. This is because $\fx$ would just be a fat point whose maximal ideal has nilpotency $n$. More specifically, in the notation of \ref{bullets}, we have $R_m = R_n$ for all $m\geq n.$
\end{remark}

\subsection{Stability}\label{2}

Let $\sH_{\kappa}=\mathbf{Gr}(\categ{Form}_{\kappa})_{\bL}$ be the localization of the Grothedieck ring of
the formal motivic site at $\bL$ and let $\sG_{\kappa} = 
\mathbf{Gr}(\mathbf{Var}_{\kappa})_{\bL}$ be the localization of the Grothendieck ring of varieties
over ${\kappa}$ at $\lef$.
We have a set-theoretic function $\mbox{dim}$ from $\sH_{\kappa}$ to $\bZ\cup\{-\infty\}$ defined by sending the element $[\sX]\lef^{-i}$ to 
the integer $\dim \sX - i$  and extending linearly through $\zet$. Here, $\sX$ is of course a formal sieve, and we set $\dim 0 = -\infty$. In exactly the same way, we have a set-theoretic function  $\mbox{dim}$ from $\sG_{\kappa}$ to $\bZ\cup\{-\infty\}$. 
 Let $\hat\sH_{\kappa}$ and $\hat\sG_{\kappa}$ be the 
group completion of $\sH_{\kappa}$ and $\sG_{\kappa}$ with respect to the filtration of subgroups given by
\begin{equation*} F^m\sH_{\kappa} = \{ X \in \sH_{\kappa} \mid \dim X < m\} \ \ \ \mbox{and} \ \ \ F^m\sG_{\kappa} =
\{X \in \sG_{\kappa} \mid \dim X < m\} \ , 
\end{equation*}
respectively. Multiplication in $\sH_{\kappa}$ and $\sG_{\kappa}$ extend to the group completions making
 $\hat\sH_{\kappa}$ and $\hat\sG_{\kappa}$ rings. 

\begin{theorem} \label{reductionring}
 There is a ring homomorphism
\begin{equation*}
\sigma:  \mathbf{Gr}(\categ{Form}_{\kappa}) \to \mathbf{Gr}(\mathbf{Var}_{\kappa}) \ .
\end{equation*}
Moreover, $\sigma$ canonically
induces ring homomorphisms $\sigma' :\sH_{\kappa} \to \sG_{\kappa}$ and $\hat \sigma : \hat
\sH_{\kappa} \to \hat\sG_{\kappa}$. The ring homomorphism $\hat\sigma$ is a continuous ring
homomorphism of topological rings. 
\end{theorem}

\begin{proof}
The case when the underlying field is algebraically closed is done in the proof of Theorem 7.7 of \cite{Sch1}. That argument
only depends on the the field being algebraically closed because Schoutens uses a weak form of Chevalley's theorem
to insure that $\sX({\kappa}) = \mbox{im}(Y\to X)({\kappa})$ is a constructible subset of $X({\kappa})$. This result was generalized by 
Grothendieck in \cite{G2} Theorem 1.8.4 to all morphisms of finite presentation between quasi-compact 
quasi-separated schemes\footnote{Here, one needs the quasi-compact and quasi-separated actually only on $X$ as $Y$ is constructible. See
(1.8.1) of \cite{G2} for the details.}. However, all morphisms in $\sch{\kappa}$ satisfy this hypothesis regardless of the ground field. 
Thus, Schoutens' proof goes through over any field ${\kappa}$.  Explicitly, we define
\begin{equation*}
 \sigma(\sX) = f^{-1}(\mbox{im}(Y\to X)({\kappa}))
\end{equation*}
where $f$ is the reduction map from $X^{red} \to X$. Note that by 1.8.2 of \cite{G2}, $\sigma(\sX)$ is indeed a constructible subset of the variety
$X^{red}$. 

The rest follows from basic facts concerning localization and
completion as functors. The could consult  \cite{AM} for these facts. 
\end{proof}

\begin{definition} We define a map of sets
$\mu : StS_{\fx} \times \Arc_{\kappa} \to \hat\sH_{\kappa}$ by
\begin{equation*}\mu(X, \fx) :=\mu_{\fx}(\nabla_{\fx}X):=  
[(\nabla_{\fx/\fn}X)^{\circ}]\bL^{-s_{\fx}(X)\dim X} \end{equation*} where $\fn
\in \bX$ is such that $\ell(\fn) = s_{\fx}(X)$. We call 
 $\mu_{\fx}(\nabla_{\fx}X)$ the {\it stable motivic} $\fx$-{\it volume of} $X$.
\end{definition}

By a {\it  function} from $\alpha: (\nabla_{\fx}X)^{\circ} \to \bN\cup\{+\infty\}$, we mean an assignment (in general not a functor) which associates to each $\fm \in\fatpoints\kappa$ a set theoretic function from $(\nabla_{\fx}X)^{\circ}(\fm) \to \bN\cup\{+\infty\}$. As $(\nabla_{\fx}X)^{\circ}$ is itself represented by a scheme, we will often write $\nabla_{\fx}X$ for $(\nabla_{\fx}X)^{\circ}$.

\begin{definition}
Let $X$ be an element of $StS_{\fx}$ (resp., $WStS_{\fx}$ and
$LStS_{\fx}$). Let $\alpha : \nabla_{\fx}X \to \bN\cup\{+\infty\}$ be a function
such that for all $\fm \in \bX$ with $\ell(\fm) \geq s_{\fx}(X)$ (resp., $\ell(\fm)\geq WStS_{\fx}$ and
$\ell(\fm)\geq LStS_{\fx}$))  the subset
$\alpha^{-1}(n)$ of  $\nabla_{\fx}X$ is such that
$\pi_{\fm}^{\fx}(\alpha^{-1}(n))$  is a formal sieve. In this case, we say that
$\alpha$ is a $\fx$-{\it stable function} (resp., $\fx$-{\it weakly stable function} and $\fx$-{\it laxly stable
function}).
\end{definition}

\begin{definition}
 Let $A$ be a subsieve of $\nabla_{\fx} X$. Suppose that  the characteristic function of $A$ defined by
$I_{A} : \nabla_{\fx} X \to \bN\cup\{+\infty\}$ defined by \[I_{A}(a) =
\begin{cases} 1 &\text{whenever $a\in A$} \\ 0 &\text{otherwise}\end{cases}\] is
$\fx$-weakly stable (resp., $\fx$-laxly stable and $\fx$-stable). In this case, we say
that $A$ is $\fx$-{\it weakly-stable} (resp., 
$\fx$-{\it laxly-stable} and $\fx$-{\it stable}).
\end{definition}

\begin{theorem} \label{subscheme}
Assume that $X$ is $\fx$-stable where $(\fx, \bX) \in \Arc_{\kappa}$. Let  $S$ be a
closed subscheme of $X$. Let $X_S$ be the formal completion of $X$ along $S$. Then, $X_S$ is $\fx$-laxly stable, $ls_{\fx}(X_S) =
s_{\fx}(X)$, and hence $\nabla_{\fx}X_S$ is a $\fx$-stable subset of
$\nabla_{\fx}X$. 
\end{theorem}
\begin{proof} 
Since $X$ is $\fx$-stable, there exists an $n$ such that $\nabla_{\fx/\fm}X \to
\nabla_{\fx/\fn}X$ is a piecewise trivial fibration for all $\fm \geq \fn$ in $\bX$
with general fiber $\bA_{{\kappa}}^{d(\ell(\fm)-\ell(\fn))}$ where $d$ is the dimension of $X$.  By Theorem 4.4 of \cite{Sch2}, we may
cover $X$ by a finite collection of opens 
$U_i$ such that $$\nabla_{\fx/\fm}U_i \cong \nabla_{\fx/\fn}U_i
\times_{\kappa}\bA_{{\kappa}}^{d(\ell(\fm)-\ell(\fn))} \ .$$
By Lemma 4.9 of loc. cit., we have that 
\begin{equation}
\begin{split}
\nabla_{\fx/\fm} (U_i)_{U_i\cap S} &\cong  (U_i)_{U_i\cap
S}\times_{\kappa}\nabla_{\fx/\fn}U_i  \times_{\kappa}\bA_{{\kappa}}^{d(\ell(\fm)-\ell(\fn))} \\
&\cong  \nabla_{\fx/\fn} (U_i)_{U_i\cap S}\times_{\kappa}\bA_{{\kappa}}^{d(\ell(\fm)-\ell(\fn))}
\end{split}
\end{equation}
Therefore, the natural morphism $\nabla_{\fx/\fm} X_{S}\to\nabla_{\fx/\fn}
X_{S}$ is a piecewise trivial fibration with general fiber $\bA_{{\kappa}}^{d(\ell(\fm)-\ell(\fn))}$
for all $\fm \geq \fn$. Thus, $X_S$ is $\fx$-laxly-stable and $ls_{\fx}(X_S) =
s_{\fx}(X)$. 
\end{proof}

\subsection{The induced measure on generalized arc spaces.}\label{firstmeasure}

Let $\bB_{\fx}^{X}[s]$ be the collection of all $\fx$-stable subsieves of
$\nabla_{\fx}X$.

\begin{theorem} \label{measuremu}
Let $X$ be $\fx$-stable. We have a set
map $$\mu_{\fx} : \bB_{\fx}^{X}[s] \to \hat\sH_{\kappa}$$ with the following
properties:
 
\noindent(a) When $A \in \bB_{\fx}^{X}$, then $$
\mu_{\fx}(A):=[\pi_{\fn}^{\fx}(A)]\bL^{-\ell(\fn)\dim X} \in \hat\sH_{\kappa} \ $$
for any $\fn\in\bX$ such that $\ell(\fn)\geq S_{\fx}(X)$.

\noindent(b) When $\{A_i\}$ is a countable collection of mutually disjoint
elements of $\Gamma_{\fx}^{X}[s]$, then we may define$$\mu_{\fx}(\cup_i A_i) := \sum_{i}
\mu_{\fx}(A_i) \  $$
whenever the right hand side converges in $\hat\sH_{\kappa}$.

\end{theorem}

\begin{proof} 
All that needs to be shown is that the definition in (b) is independent of choice of
presentation. Thus, let $A$ be equal to $\sqcup_i A_i$ and $\sqcup_i B_i$ where $A_i, B_i \in
\bB_{\fx}^{X}[s]$ with respective convergent summations. Then, 
\begin{equation}\begin{split}
 \sum_{i} \mu_{\fx}(A_i )
&=\sum_{i}\mu_{\fx}(\sqcup_j(A_i\cap B_j) )\\
&=\sum_{i} \sum_{j}\mu_{\fx}(A_i\cap B_j )
\\
&=\sum_{j} \mu_{\fx}(\sqcup_i(A_i\cap B_j))
\\
&= \sum_{j} \mu_{\fx}(B_j ) \ .
\end{split}\end{equation}

\end{proof}

We have then the following definition of the geometric schemic integral:

\begin{definition} 
 Let $\alpha : \nabla_{\fx}X \to \bN\cup\{+\infty\}$ be any $\fx$-stable function,
then we define $$\int_{\nabla_{\fx}X} \bL^{-\alpha}d\mu_{\fx}:= \sum_{n\in\bN}
\mu_{\fx}(\alpha^{-1}(n))\bL^{-n} \ .$$ We say that $\alpha$ is $\fx$-{\it
integrable}  if this summation converges in $\hat\sH_{\kappa}$. 
\end{definition}

Lets assume now that $X\in\sch\kappa$ is $\fx$-stable and generically smooth. This means that the singular locus $S$ of $X$ is a subscheme of $X$ of positive codimension and that $\nabla_{\fx}X_S$ is a $\fx$-stable subsieve of $\nabla_{\fx}X$. Moreover, 
$\nabla_{\fx}X\setminus \nabla_{\fx}X_S = \nabla_{\fx} (X\setminus X_S)$. 
This is Equation 42 of \cite{Sch2}. Moreover, by Proposition 7.2 of loc. cit., the right hand side is equal to $\nabla_{\fx}U$ where $U$ is a dense open and smooth subscheme of $X$. Finally, given a $\fx$-stable subsieve $A$ of $\nabla_{\fx}X$, we have that $A\cap \nabla_{\fx}U$ is $\fx$-stable. Following this and the work of \cite{DL1}, we define an augmented measure $\gamma_{\fx}$ on $\bB_{\fx}^{X}[s]$ by 
\begin{equation}
\gamma_{\fx}(A) := \mu_{\fx}(A\cap\nabla_{\fx}U) \ .
\end{equation}
Moreover, since $U$ is stable at level $0$, it is immediate that 
\begin{equation}
\gamma_{\fx}(A) = [\pi_{\spec\kappa}^{\fx}(A)\cap U]\bL^{-\dim X} \ .
\end{equation}

\begin{theorem} \label{samesame}The measures $\mu_{\fx}$ and $\gamma_{\fx}$ are the same when $X$ is generically smooth.
\end{theorem}
\begin{proof}
The proof below is a straightforward  adaptation of the proof for the analogous statement in \cite{DL1}.
Note that I am claiming that for any $A \in\bB_{\fx}^{X}[s]$,
\begin{equation}\label{propmu}
\mu_{\fx}(A) = \mu_{\fx}(A\setminus\nabla_{\fx}X_S) \ .
\end{equation}
Now, for ease of notation we will denote by $\pi_m$ the canonical morphism
$\pi_{\fm}^{\fx} : \nabla_{\fx} X \to \nabla_{\fm}X.$ 
We have the following partition of $A\setminus\nabla_{\fx}X_S$ 
\small
$$(A\setminus\pi_{m}^{-1}(\pi_m(\nabla_{\fx}X_S)))\sqcup\bigsqcup_{n\geq
m}((\pi_{n}^{-1}(\pi_n(\nabla_{\fx}X_S))\setminus\pi_{n+1}^{-1}(\pi_{n+1}
(\nabla_{\fx}X_S)))\cap A) \ ,$$
\normalsize
and we have the following partition of $A$
\small
$$
A\setminus\pi_{m}^{-1}(\pi_m(\nabla_{\fx}X_S))\sqcup\pi_{m}^{-1}(\pi_m(\nabla_{
\fx}X_S))\cap A \ .$$
\normalsize
Therefore, the difference between $\mu_{\fx}(A\setminus\nabla_{\fx}X_S)$ and
$\mu_{\fx}(A)$ is 
\begin{equation*}
\sum_{n\geq m}\mu_{\fx}(\pi_{n}^{-1}((\pi_n(\nabla_{\fx}X_S))\cap A) -
\sum_{n\geq m}\mu_{\fx}(\pi_{n}^{-1}(\pi_{n}(\nabla_{\fx}X_S)))\cap A)) \ ,
\end{equation*}
which is equivalent to $0$ in $\hat\sH$. 

Now, we can prove that this is
independent of
choice of closed sub-scheme $S$. Let $A\in\bB_{\fx}^{X}[s]$ and let $S'$ be another
closed sub-scheme of $X$ such that $\dim S' < \dim X$.  Using Equation
\ref{propmu}, we
have
\begin{eqnarray}
 \mu_{\fx}(A\setminus\nabla_{\fx}X_S) &=& 
\mu_{\fx}(A\setminus\nabla_{\fx}X_S\setminus\nabla_{\fx}X_{S'}) \nonumber\\
 &=& \mu_{\fx}(A\setminus\nabla_{\fx}X_{S'}\setminus\nabla_{\fx}X_S)
\nonumber\\
 &=&\mu_{\fx}(A\setminus\nabla_{\fx}X_{S'})\ . \nonumber
 \end{eqnarray}  
\end{proof}

Thus, when $X$ is not generically smooth it is not obvious how to relate $\mu_{\fx}$ and the classical motivic measure. This is one of the main reasons why we find the study of schemes whose reductions are smooth so interesting in this context. For example, the fat point $\fl_2$ whose coordinate ring is the dual numbers is not generically smooth yet its reduction is smooth. We will show that we can use $\mu_{\fl}$ to measure $\fl_2$. However, the lack of generic smoothness means that this is not just a simple analogue of classical motivic integration. In other words, one cannot say that invariants of $\sch\kappa$ are essentially the same as invariants of $\mathbf{Var}_\kappa$. Moreover, what will become clear is that wildly different technology must be developed in order to understand the additive invariants of $\sch\kappa$.

\subsection{Change of variables formula.}\label{changesection}

For $A \in \bB_{\fx}^{X}$, we say that it is a {\it constructible} if the formal sieve $\pi_{\fn}^{\fx}(A)$ is a constructible cone (cf., \S 2 and \S 7 of \cite{Sch2}) for all $\fn$ such that $\ell(\fn)\geq S_{\fx}(X)$.
The collection $\bB_{\fx}^{X}[s]$ contains the boolean algebra $\Gamma_{\fx}^{X}[s]$  of all constructible $\fx$-stable subsieves. Note that every $\nabla_{\fx}X_S$ is always constructible. We say that a $\fx$-stable function $\alpha$ is {\it constructible} if its fibers are constructible.

Let $f: X \to Y$ be a birational morphism in $\sch\kappa$ where $X$ is generically smooth.
Then,  $\ord f^*\Omega_{Y}^{d}$ is defined exactly as in \cite{DL1}. For simplicity of notation, we write $J_X[f]$
in place of $\ord f^*\Omega_{Y}^{d}$. If $X$ is $\fl$-stable, then $J_X[f]$ is a constructible $\fl$-stable function.

\begin{theorem} \label{changemu} Let $X\in\sch{\kappa}$ be of  pure
dimension $d$ where ${\kappa}$ is of characteristic zero.
 Let $f : X \to Y$ be a
proper birational morphism, $X$ be generically smooth,
 and let $\alpha : \nabla_{\fl}Y \to
\bZ\cup\{+\infty\}$ be a constructible $\fl$-stable function, then we have the following
identity
\begin{equation*}
\int_{\nabla_{\fl}Y} \bL^{-\alpha}d\mu_{\fl} = \int_{\nabla_{\fl}X} 
\mathbb{L}^{-\alpha\circ f - J_X[f]}d\mu_{\fl}
\end{equation*}
 in $\hat \sH_{\kappa}$
when both sides converge.
\end{theorem}

\begin{proof}
First note that since $X$ and $Y$ are generically smooth, $\mu_{\fl} = \gamma_{\fl}$. This was shown in Theorem \ref{samesame}.
By construction $\gamma_{\fl}$ reduces (via $\hat\sigma$) to the measure in \cite{DL1}. Thus the statement is more or less immediate then as soon as we prove that we can lift equations in $\hat\sG_{\kappa}$ in a suitable sense. This is done in Lemma \ref{triangle}. With this in mind, we simply have that $\hat\sigma$ is a continuous morphism of topological rings so that under $\hat\sigma$ so that the reductions of each side coverge. Therefore,   by Lemma 3.4 of \cite{DL1}, we have the following equality
\begin{equation}\label{regularchangeofvariables}
\int_{\nabla_{\fl}Y} \bL^{-\beta}d\mu_{\fl} = \int_{\nabla_{\fl}X} 
\mathbb{L}^{-\beta\circ f - \ord f^{*}\Omega_{Y}^{d}}d\mu_{\fl}
\end{equation} 
in $\hat\sG_{\kappa}$. Applying the lifting lemma (Lemma \ref{triangle}) gives the result.
\end{proof}

\section{Applications of schemic integration}

In this section, we push the theory further and investigate a couple of applications. 

\subsection{Local deformations and stability.}\label{5}
First, let us take the definition of smoothness in Chapter III, Section 10 in \cite{Ha1}. That is we will assume that $f : X\to Y$ is a smooth morphism in $\sch\kappa$ implies that it is of relative dimension $n$. In particular, that is 
if $X$ is smooth over a finite $\kappa$-scheme, then all the irreducible components of $X$ have the same dimension $d$. Said another way, $X$ will automatically be of pure dimension $d$. Although, results easily generalize to the definition of smoothness in EGA, it simplifies the statement of most theorems substantially to include relative dimension as part of the definition of smoothness.

Now, the prototypical example of a stable scheme is any smooth scheme $X$ in $\sch{\kappa}$
of dimension $d$. One of the goals of this chapter is to generalize the following theorem. 
\begin{theorem} \label{smoothvolume}
For any $\fx \in \Arc_{\kappa}$ and any $X\in\categ{SmSch}_{\kappa}$ of dimension $d$, we have
\begin{equation*}
\mu_{\fx}(\nabla_{\fx}X) = [X]\bL^{-d} \ .
\end{equation*}
\end{theorem}
\begin{proof}
 This follows immediately from Theorem \ref{smoothstabletheorem}.
\end{proof}

\begin{example}
 By Theorem 4.8 of \cite{Sch2},  $\nabla_{\fx} \Spec {\kappa} \cong \fx$. Thus,  we can use Theorem \ref{smoothvolume} to calculate the volume of any admissible 
 arc. For every $\fx \in \Arc_{\kappa}$, we have
 \begin{equation*}
  \mu_{\fx}(\fx) := \mu_{\fx}(\nabla_{\fx} \Spec {\kappa}) = [\Spec {\kappa}] \bL^{-0} = 1
 \end{equation*}
 Clearly, $\nabla_{\fx} \emptyset = \emptyset$, from which we may obtain $\mu_{\fx}(\emptyset) = 0 $. 
 Finally, note that $\mu_{\fx}(\nabla_{\fx}\bA_{{\kappa}}^{d})  = 1$ as well. 
\end{example}

\begin{proposition} \label{standardgroup}
Assume that ${\kappa}$ is of characteristic $0$. 
 The set map 
 \begin{equation*}
 \mu_{\fx} \circ \nabla_{\fx} : \mathbf{Var}_{\kappa} \to \sH_{\kappa}
 \end{equation*}
 induces a ring homomorphism from  $M_{\fx}: \mathbf{Gr}(\mathbf{Var}_{\kappa}) \to \sH_{\kappa}$.
\end{proposition}
\begin{proof}
 We note that by resolution of singularities every element $[S]$ of $\mathbf{Gr}(\mathbf{Var}_{\kappa})$ can be written as a finite integral sum of equivalence 
 classes of smooth connected varieties $[X_i]\in \mathbf{Gr}(\mathbf{Var}_{\kappa})$. Thus, we may define
 \begin{equation*}
 M_{\fx}([S]) := M_{\fx}(\sum_{i=1}^{l} m_i[X_i]) := \sum_{i=1}^{l}m_i \mu_{\fx}( \nabla_{\fx}X_i) = \sum_{i=1}^{l}m_l[X_i]\bL^{-\dim(X_i)}\ .
\end{equation*}
 \end{proof}

\begin{remark} \label{SBremark}
Let $SB_{\kappa}$ be the set of stably projective varieties over ${\kappa}$ where ${\kappa}$ is of characteristic $0$. We can extend
 the ring homomorphism above to a ring homomorphism from $\bZ[SB_{\kappa}]$ to $\sH_{\kappa}$ by defining
 $\bar M_{\fx}([T]) := \bL \cdot M_{\fx}([S])$
 where $[T]$ is in $\bZ[SB_{\kappa}]$ and $[S] \in \mathbf{Gr}(\mathbf{Var}_{\kappa})$ is such that $[S] \  \mbox{mod} \ (\bL) = [T]$ (cf. \cite{LL}).
 Note that
 \begin{equation*}
 \Ker(\bar M_{\fx}) = \mathbf{Gr}_{0}(\mathbf{Var}_{\kappa})/(\bL) \cong \bZ[\bP^1] \ .
 \end{equation*}
\end{remark}

\begin{remark} 
 Theorem \ref{smoothvolume}, Proposition \ref{standardgroup}, and Remark \ref{SBremark} all display the fact that
 our construction of the schemic motivic volume partly generalizes the classic notion of geometric motivic integration
 theory (cf. \cite{DL1}, \cite{DL2}). Being able to integrate over different admissible arcs becomes a great benefit 
 when we pass from the category of varieties over ${\kappa}$
 to schemes over ${\kappa}$ as we will see below. 
 Moreover, extending the above ring homomorphisms $M_{\fx}$ is possible. 
We discuss this in \S \ref{8}
 and \S \ref{9} as a way to display the fact that many results obtained in the classical theory will have analogous results in the schemic theory.
\end{remark}

\begin{definition} We say that a scheme $Y$ is a {\it local deformation} of a scheme
$X$ if there exists a fat point $\fn$ which admits a flat morphism $Y\to \fn$
together with a morphism $X \to Y$ such that the induced morphism $X \to
Y\times_{\fn} {\kappa}$ is an isomorphism. 
\end{definition}

If $X$ is smooth, then it is known (cf. \cite{Ha2} p. 38-39) that every local
deformation $Y$ is trivial -- i.e., $Y \cong X \times_{\kappa} \fn$ for some fat point
$\fn$. Conjecturally, it seems likely that such a scheme $Y$ would also be
$\fx$-stable for suitable choices of admissible arcs $\fx$.  We may generalize even further because the local deformation $Y$ will be smooth over
the fat point $\fn$ (just apply base change by $\fn$). Therefore, we have a
potential source for a plethora of stable schemes -- that is, {\it schemes
which are smooth over a fat point}. 
Thus, it is natural to postulate the following conjecture.

\begin{conjecture}\label{conjone}
Every scheme $X$ which is smooth over a fat point $\fn$  is $\fx$-stable where
$\fx$ is any admissible arc in $\Arc_{\kappa}$ containing $\fn$. 
\end{conjecture}

Example \ref{nofibration} below shows that the conjecture above in full
generality is
false. There seems to be an obstruction to having linear fibers based on the
type of fat point $\fn$ for which the scheme is smooth over. In what follows, we
will attempt to characterize this problem more clearly.

\begin{lemma} \label{etalefiber}
Let $f: X \to Y$ be an \'{e}tale morphism of schemes and let $\fn$
be a fat point over ${\kappa}$. Then $$\nabla_{\fn} X \cong
X\times_Y
\nabla_{\fn}Y$$ \end{lemma}

\begin{proof} Cf. Theorem 4.12 of \cite{Sch2}. \end{proof}

\begin{lemma} Let $f: X \to Y$ be an \'{e}tale morphism of schemes and let $\fx$
be a the limit of a point system (e.g., $\fx$ is an
admissible arc over ${\kappa}$), then $$\nabla_{\fx} X \cong X\times_Y
\nabla_{\fx}Y$$ \end{lemma}

\begin{proof} This follows directly from the definition of inverse limit, the previous lemma, and
Lemma 7.4 of \cite{Sch2}. \end{proof}

\begin{theorem} Let $X \to \fn$ be a smooth morphism and let $\fm$ be any fat point.
The canonical morphism $\nabla_{\fm} X \times_{\kappa} \bA_{\fn}^{d}
\to X$ is a piecewise trivial fibration with fiber $\nabla_{\fm} \bA_{\fn}^{d}$ where $d=\dim X$.
\end{theorem}

\begin{proof}
Using Theorem 4.4 of \cite{Sch2}, we may cover $\nabla_{\fm}X$ by opens
$\nabla_{\fm}U$ where $U$ is an open in $X$. Therefore, by shrinking $X$ if
necessary, we may assume that there is an \'{e}tale morphism $X\to
\bA_{\fn}^{d}$
(cf., \cite{Liu} Chapter 6, Corollary 2.11).

We apply Lemma \ref{etalefiber} to obtain 
\begin{equation*}
\nabla_{\fm} X \cong X\times_{\bA_{\fn}^{d}} \nabla_{\fm} \bA_{\fn}^{d}\cong
X\times_{\kappa}
\bA_{{\kappa}}^{d\ell(m)} \times_{\bA_{\fn}^{d}} \nabla_{\fm} \fn \ .
\end{equation*}
Therefore, 
\begin{equation*}
\nabla_{\fm}X \cong X\times_{{\kappa}} \bA_{{\kappa}}^{d(\ell(\fm)-1)}\times_{\fn}\nabla_{\fm}
\fn \ .
\end{equation*}
Thus, 
\begin{equation*}
\nabla_{\fm} X \times_{\kappa} \bA_{\fn}^{d} \cong X\times_{{\kappa}}
\bA_{{\kappa}}^{d\ell(\fm)}\times_{\fn}\fn \times_{\kappa}\nabla_{\fm} \fn \cong
X\times_{{\kappa}}\nabla_{\fm} \bA_{\fn}^{d} \ .
\end{equation*}
\end{proof}

\begin{corollary} Let $X \to \fn$ be a smooth morphism and let $\fm$ be any fat point.
There is an isomorphism 
\begin{equation*}
\nabla_{\fm} X \times_{\kappa} \bA_{\fn}^{d} \cong X \times_{\kappa} \nabla_{\fm}
\bA_{\fn}^{d} \ . 
\end{equation*}
\end{corollary}

\begin{proof} This follows because we may cover $\nabla_{\fm}X$ with opens of the form $\nabla_{\fm}U$ as noted at the beginning of the proof of the previous theorem.
\end{proof}

Now, we wish to grapple with the general conjecture. Our approach is to ask
questions about lifts to $X \to \fm$ of a given smooth morphism $X \to \fn$ when
$\fn \inj \fm$ closed immersion of fat points over $\kappa$. 
Note that a closed immersion $\fn \into \fm$ is only an open immersion of schemes if it is an isomorphism even though it is trivially an open immersion
of the underlying topological spaces.

\begin{theorem} \label{liftingtheorem}
Suppose that $X$ is affine.  Let $f: X \to \fn$ be a
smooth morphism where $\fn\in\fatpoints\kappa$ and  let $\iota: \fn \inj \fm$ be a closed immersion in $\fatpoints\kappa$. Then, there exists a smooth morphism $\bar f: X \to \fm$ such that $\bar f = \iota\circ f$.
\end{theorem}

\begin{proof} First, we may reduce to the case where the closed immersion $\fn
\inj \fm$ is given by a square zero ideal $J$. Then it is well-known that the
obstruction to lifting smoothly to $ X \to \fm $ lies in $$H^{2}(X, T_{X}
\otimes \tilde J)$$ where $T_X$ is the tangent bundle of $X$. Since
$T_{X}\otimes \tilde J$ is quasi-coherent and $X$ is assumed to be affine, we
have that 
\begin{equation*}
H^{2}(X, T_{X} \otimes \tilde J) = 0 \ ,
\end{equation*}
by Theorem 3.5 of Chapter III of \cite{Ha1}.
\end{proof}

This lemma states that any affine  $X\in\sch{\kappa}$ with $X \to \fn$ smooth has a
well-defined (quasi-smooth\footnote{Without going into the details, we just
mention that, for us, this means that the k\"{a}hler differentials
$\Omega_{X/\fx}$
are a finite, locally-free $\sO_{\fx}$-module.}) morphism (of locally ringed
spaces) $X \to \fx$ where $\fx \in \Arc_{\kappa}$ is such that $\fx$ contains $\fn$.
Therefore,  we may consider the following refinement of Conjecture
\ref{conjone}.

\begin{conjecture} \label{conjtwo}
Let $X\in\sch\kappa$ be affine.
Assume that $X \to \fn$ and $X\to\fm$ are smooth morphisms such
that $\fn \inj \fm$ are elements of the point system $\bX$
associated to an admissible arc $\fx\in\Arc_{\kappa}$. Then the natural morphism
$$\pi_{\fn}^{\fm} : \nabla_{\fm}X \to \nabla_{\fn}X $$ is a piecewise trivial
fibration over ${\kappa}$ with general fiber 
$\bA_{{\kappa}}^{d(\ell(m)-\ell(n))}.$ \end{conjecture}

Naively, one might proceed to prove Conjecture \ref{conjone} and Conjecture
\ref{conjtwo} in the
following way. Let $d = \dim X$. We may cover $X$ by a finite numbers of opens
$U$ with \'{e}tale morphisms $U \to \bA_{\fm}^{d}$ . As open 
immersions are smooth, the restriction $U \to \fn$ of $X \to \fn$ is also
smooth. Therefore, we may cover $U$ by opens $V$ with 
\'{e}tale morphisms $V \to \bA_{\fn}^{d}$. Thus, from the start, we may assume
that we have \'{e}tale morphisms $X \to 
\bA_{\fm}^{d}$ and $X\to\bA_{\fn}^{d}$. We then have the following isomorphisms:
\begin{equation*}
\nabla_{\fm}X\cong X 
\times_{\bA_{\fm}^{d}}\nabla_{\fm}\bA_{\fm}^{d} \ \  \mbox{and}  \ \
\nabla_{\fn}X\cong X 
\times_{\bA_{\fn}^{d}}\nabla_{\fn}\bA_{\fn}^{d} \ .
\end{equation*}
Furthermore, we have the commutative diagram
\[\begin{CD}
\nabla_{\fm}X\>{\cong}>> (X \times_{\fm}\nabla_{\fm}\fm)\times_{\kappa}
\bA_{{\kappa}}^{d(\ell(m)-1)} \\
\V{\pi_{\fn}^{\fm}}VV \V{}VV \\
\nabla_{\fn}X\>{\cong}>> (X \times_{\fn}\nabla_{\fn}\fn)\times_{\kappa}
\bA_{{\kappa}}^{d(\ell(n)-1)}
\end{CD}\]
This reduces to understanding the morphism $\nabla_{\fm}\fm \to
\nabla_{\fn}\fn$. 

It is exactly here that this approach breaks down because understanding the
morphism $\nabla_{\fm}\fm \to \nabla_{\fn}\fn$ is elusive to us. In fact, there
can be a complicated scheme structure on the so-called auto-arcs
$\nabla_{\fn}\fn$. This is discussed in much more detail in \cite{Sch2}. We will
now provide a counter-example to the claim that $\nabla_{\fm}\fm \to
\nabla_{\fn}\fn$ or $(\nabla_{\fm}\fm)^{red} \to (\nabla_{\fn}\fn)^{red}$ is
always a piece-wise trivial fibration.

\begin{definition} \label{defsimplepoint}
We say that a fat point $\fn$ is {\it simple} if
$(\nabla_{\fn}\fn)^{red} \cong \mathbb{A}_{{\kappa}}^{m}$ for some $m\geq 0$.  We say
that a point system is {\it simple} if all of its fat points are simple and that a point system $\bX$ is {\it eventually simple}
if $\fn \in \bX$ is simple when $\ell(\fn) >>0$. 
\end{definition}

\begin{example} \label{nofibration} In the case where $\bX = \{ \fl_m\}$, it is
standard (cf. \cite{DL1})
that our proof will work and $X$ will be $\fl$-stable. Consider the case where
$\fm = \Spec R$ where $R = {\kappa}[x,y]/(x^2,xy,y^2)$ and $\fn = \Spec {\kappa}$ so that
$\nabla_{\fn}\fn = {\kappa}$. A quick calculation shows that 
\begin{equation*}
\nabla_{\fm}\fm = \Spec {\kappa}[a_1, a_2, b_1, b_2, c_1, c_2]/I 
\end{equation*}
where $I$ is the ideal generated by the elements $$\{a_1a_2, a_1b_2+a_2b_1,
a_1c_2+a_2c_1, 
a_{1}^{2}, a_1b_1, a_1c_1, a_{2}^{2}, a_2b_2,a_2c_2\}\ .$$ 
It is easy to check that $\nabla_{\fm}\fm \cong \fm \times_{\kappa} \bA_{{\kappa}}^{3}$ so that 
$(\nabla_{\fm}\fm)^{red} \cong \bA_{{\kappa}}^{3}$.
In Example 4.17 of \cite{Sch2}, Schoutens found the following counter-example:
\begin{equation*}
 \fm = \Spec A/\mathcal{M}^4 , \quad A := {\kappa}[x,y]/(y^2 -x^3), \ \mathcal{M} = (\bar x, \bar y)
\end{equation*}
where $\bar x$ and $\bar y$ are the residue classes of $x$ and $y$ in $A$. In other words, $\fm$ is the $4$th order jet of 
the cuspidal curve at the origin. This is denoted in loc. cit. as $J_{O}^{4} C$ where $C = \Spec A$. 
It is argued there that $(\nabla_{\fm}\fm)^{red}$
is singular. Thus, $\fm$ is not a simple point. We can use this example to speculate that not every point system is eventually 
simple as well. Let $\fx = \fm\times_{\kappa} \Spec {\kappa}[[t]]$ and let $(\fx, \bX) \in \Arc_{\kappa}$ where 
$\bX = \{ J_{0}^{n} \bA_{\fm}^{1} \}$. Thus, $\fn\in \bX$ is such that $\fn = \fm \times_{\kappa} \spec{\kappa[t]/(t^n)}$. 
\begin{equation*}
\nabla_{\fn}\fn \cong \nabla_{\fn}\fm \times_{\kappa} \nabla_{\fn}\fl_n \ .
\end{equation*}
 Thus, it seems unlikely that $(\nabla_{\fn}\fn)^{red}$ is smooth for large $\ell(\fn) = n + \ell(\fm)$. However, I do not have a proof of this fact.
\end{example}

\begin{question}
Let $X$ be an affine scheme of finite type over an algebraically closed field and let $p$ be a closed point of $X$.
 Under what conditions on $X$ will the point system $\{J_{p}^{n}X\}$ be eventually simple?
\end{question}

With all this in mind, it becomes clear that it is necessary to add additional
hypotheses to Conjecture \ref{conjone} and Conjecture \ref{conjtwo}. Using Definition \ref{defsimplepoint}, we obtain
the following
theorem:

\begin{theorem} Let $X\in\sch\kappa$ be affine. Assume that $X \to \fn$ and $X\to\fm$ are smooth morphisms. Assume that there is a closed immersion of schemes 
$\fn \inj \fm$ where $\fm$ and $\fn$ are elements of the simple point system $\bX$
associated to an admissible arc $\fx\in\Arc_{\kappa}$. Then, the natural morphism
$$(\pi_{\fn}^{\fm})^{red} : (\nabla_{\fm}X)^{red} \to (\nabla_{\fn}X)^{red} $$
is a piecewise trivial fibration over ${\kappa}$ with general fiber 
$$\bA_{{\kappa}}^{d(\ell(\fm)-\ell(\fn))+r(\fm,\fn)}$$ where $r(\fm,\fn)$ is some
non-negative integer depending on the lengths of $\fm$ and $\fn$ and where $d = \dim X$.
\end{theorem}

\begin{proof} From our work in trying to prove Conjecture \ref{conjtwo}, we
reduce to the
case of understanding the morphism 
\begin{equation*}
(\nabla_{\fm} \fm)^{red} \to (\nabla_{\fn} \fn)^{red} \ .
\end{equation*}
Here, we use the hypothesis that the point system is simple to conclude that
this map is a piecewise trivial fibration with fiber $\bA_{\kappa}^{r}$ where $r=r(\ell(\fm),\ell(\fn))$
conceivably depends on the lengths of $\fm$ and $\fn$. 
\end{proof}

\begin{theorem} \label{endsec5}
Let $X \in \sch{\kappa}$ be such that $X \to \fn$ is smooth for some fat
point $\fn$. Further, assume that $\fn$ belongs to an eventually simple point system $\bX$
and let $\fx = \colim \bX$. Then, $X$ is rationally\footnote{See \S \ref{8}
below.}
$\fx$-laxly-stable. \end{theorem}

\begin{proof}
 Using Theorem 4.4 of \cite{Sch2}, any cover of $X$ by opens $U_i$ gives a cover
of $\nabla_{\fm}X$ by $\nabla_{\fm}U_i$ where $\fm$ is an arbitrary fat point.
Therefore, we may reduce to the case where $X$ is affine. Then, the result
follows from Theorem \ref{liftingtheorem} and Theorem \ref{endsec5}.
\end{proof}

\subsection{Smooth reductions, local deformations, and stability.}\label{6}

We would like to investigate when we can find an admissible arc $\fx\in \Arc_{\kappa}$
such that $X$ is $\fx$-laxly-stable under the assumption that $X^{red}$ is
smooth. We know that in general smoothness does not descend via a faithfully
flat morphism; however,  we have the following:
\begin{proposition} \label{egaprop}
Let $f : X \to Y$ and $h : Y' \to Y$ be two morphisms in $\sch{\kappa}$.
Let $X' =X\times_Y Y'$ and let $f' : X' \to Y'$ be the canonical projection. Suppose
further that $h$ is quasi-compact and faithfully flat, then $f$ is smooth if and
only if $f'$ is smooth. \end{proposition} 
\begin{proof}
This is a special case of Proposition 6.8.3 of \cite{G}. \end{proof}

 Let $X \in \sch{\kappa}$ be affine and write $X = \Spec A$. Choose a minimal system
of generators $g_1,\ldots, g_s$ of the nilradical $nil(A)$ of $A$. Let $x_1,
\ldots, x_s$ be $s$ variables and let $J$ be the kernel of the map from
${\kappa}[x_1,\ldots,x_s]$ to $A$ which sends $x_i$ to $g_i$. We set $R :=
{\kappa}[x_1,\ldots, x_s]/J$. Then, $R\inj A$ and we have the following:

\begin{lemma} \label{algebralemma}
Let $X$ be a connected affine scheme in $\sch{\kappa}$, set $\fn = \Spec
R$ where $R$ is as in the previous paragraph, and let $l$ be any positive
integer. Then, $\fn$ is fat point over ${\kappa}$, and we have the following
decompositions:

\noindent (a) \quad $X^{red} \cong X\times_\fn \Spec {\kappa}$ 

\smallskip

\noindent (b) \quad $X \times_{\bA_{\fn}^{d}}\bA_{{\kappa}}^{dl} \cong X^{red}\times_{\kappa}
\bA_{{\kappa}}^{d(l-1)}$.
\end{lemma}
\begin{proof}
Write $X = \Spec A$ for some finitely generated ${\kappa}$-algebra. It is basic that $R
\inj A$. Let $\sM = (x_1,\ldots, x_s)R$. Clearly, 
$\sM$ is a maximal ideal of $R$.  
Moreover, $\sM A \subset nil(A)$ by
construction. Therefore, there exists an $N$ such that 
$\sM^N = 0$. Thus, $R$ is artinian ring with residue field $\kappa$.
We assumed $X$
was connected so that $R$ would be local. 
Indeed, by injectivity of $R \inj A$, any direct sum decomposition of $R$ would
immediately imply a direct sum decomposition of $A$ 
as it would entail that $R$ (and hence $A$) contains orthogonal idempotents $e_1
\neq e_2$.  

Note that the containment $\sM A \subset nil(A)$ is actually an equality by
construction. Now, use the fact that ${\kappa} = R/\sM $ so 
that $$A \otimes_R {\kappa} \cong A\otimes_R (R/\sM) \cong (A/\sM A) \otimes_R R \cong
A/\sM A \cong A/nil(A)$$ where the second 
isomorphism is a well-known property of tensor products for $R$-algebras. This
proves part (a). 

Part (b) is really a restatement of the work done in the preceding paragraph. 
One should just note that $$ X\times_{\bA_{\fn}^{d}} \bA_{{\kappa}}^{dl} \cong
X\times_{\bA_{\fn}^{d}} {\kappa} \times_{\kappa} \bA_{{\kappa}}^{dl} $$
so that we can apply (a) to the right hand side to obtain 
$$X\times_{\bA_{\fn}^{d}} {\kappa} \times_{\kappa} \bA_{{\kappa}}^{dl} \cong X^{red}
\times_{\bA_{{\kappa}}^{d}}\bA_{{\kappa}}^{dl} \ .$$
This proves the result part (b). 
\end{proof}

\begin{theorem} Let $X$ be a connected affine scheme. Then,  $X^{red}$ is smooth if
and only if there exists a smooth morphism $X \to \fn$ where $\fn = \Spec R$ is
as in the previous lemma. 
\end{theorem}

\begin{proof}
This is just a restatement of Proposition \ref{egaprop} where $Y' = \Spec {\kappa}$, $Y
= \fn$,
and $Y' \to Y$ is the canonical morphism. Indeed, by Lemma \ref{algebralemma},
$X' := X\times_Y
Y' \cong X^{red}$, and the homomorphism of rings $R \to {\kappa} $ given by modding out
by $\sM$ is both surjective and flat.
\end{proof}

\begin{corollary} Let $X$ be a connected affine scheme with $X^{red}$ smooth and  let
$\fn = \Spec R$ be as before. Then, for any eventually simple point system $\bX$ containing $\fn$, $X$ is rationally
$\fx$-laxly-stable where $\fx=\colim\bX$.
\end{corollary}

\begin{proof} The reader should refer to \S \ref{8} for the definition of rationally $\fx$-laxly stable. This result follows immediately from Theorem \ref{endsec5} and the
previous
theorem. 
\end{proof}

\subsection{Non-affine smooth reductions and stability.}\label{7}

\begin{theorem} \label{startsec7}
Let $X \in \sch{\kappa}$ be such that $X^{red}$ is smooth. Let $d$ be the dimension of $X$. Let $\fm$ be any fat point.
Then, there exists a finite cover $\{U_i\}$ of $X$ by connected open affines
such that for each $i$ 
$$\nabla_{\fm}U_i \cong U_{i}^{red} \times_{\kappa} \bA_{{\kappa}}^{d(\ell(\fm)-1)} \times_{\kappa}
\nabla_{\fm}\fn_i \ ,$$
where $\fn_i$ is the fat point associated to $U_i$ as in Theorem 6.2. Moreover,
if for each $i$ there exists an isomorphism $$(\nabla_{\fm} \fn_i)^{red} \cong
\bA_{{\kappa}}^{r} \ , $$
then the canonical morphism $$(\nabla_{\fm}X)^{red} \to X^{red}$$
is a piece-wise trivial fibration over ${\kappa}$ with general fiber
$\bA_{{\kappa}}^{d(\ell(\fm)-1)+r}$ . 
\end{theorem}

\begin{proof} We may cover $X$ by a finite collection of connected open affines
$\{U_i\}$. Since the open immersions $U_{i}^{red} \to X^{red}$ are smooth, each
$U_{i}^{red}$ is a smooth connected affine scheme. By Theorem \ref{etalefiber},
for each $i$,
there exists a fat point $\fn_i$ such that $U_i \to \fn_i$ is smooth and, by
part (a) of
Lemma \ref{algebralemma},  such that  $$U_{i}^{red}\cong U_i \times_{\fn_i}
\Spec {\kappa} \ .$$
We may shrink each $U_i$ (if necessary) so that there is an \'{e}tale morphism
$$U_i \to
\bA_{\fn_i}^{d}$$ and so that the open sets $\nabla_{\fm} U_i$ cover
$\nabla_{\fm} X$ (here, $\fm$ is any fat point over ${\kappa}$). Therefore, for each
$i$,
\begin{equation*} 
\nabla_{\fm} U_i \cong  U_i \times_{\bA_{\fn_i}^{d}}
\nabla_{\fm}\bA_{\fn_i}^{d} 
\cong  U_i
\times_{\bA_{\fn_i}^{d}}\bA_{\fn_i}^{d\ell(\fm)}\times_{\kappa}\nabla_{\fm}\fn_i  \ ,
\end{equation*}
where the first equation is a consequence of Lemma \ref{etalefiber} and the
second equation
follows from  Proposition 3.3 and Theorem 4.8 of \cite{Sch2}. Using part (b) of
Lemma \ref{algebralemma}, we obtain
\begin{equation*}
\nabla_{\fm}U_i \cong U_{i}^{red} \times_{\kappa} \bA_{{\kappa}}^{d(\ell(\fm)-1)} \times_{\kappa}
\nabla_{\fm} \fn_i \ .
\end{equation*}
Because the first two factors in the above fiber product on the right hand side
are affine, it is a straightforward verification that
\begin{equation*}
(\nabla_{\fm}U_i)^{red}  \cong U_{i}^{red} \times_{\kappa} \bA_{{\kappa}}^{d(\ell(\fm)-1)}
\times_{\kappa} (\nabla_{\fm} \fn_i)^{red} \ ,
\end{equation*}
from which the theorem follows. 
\end{proof}

\begin{theorem} 
Let $X \in \sch{\kappa}$ be such that $X^{red}$ is smooth.  Let $d$ be the dimension of $X$. Let $\fm$ be any fat point. Then, there
exists a finite collection of fat point $\fn_i$ such that

\noindent (a) \quad $\nabla_{\fm}X \cong X^{red} \times_{\kappa}
\bA_{{\kappa}}^{d(\ell(\fm)-1)} \times_{\kappa} \nabla_{\fm}(\sqcup_i \fn_i) \  $

\smallskip 

\noindent (b) \quad $(\nabla_{\fm}X)^{red} \cong X^{red} \times_{\kappa}
\bA_{{\kappa}}^{d(\ell(\fm)-1)} \times_{\kappa} (\nabla_{\fm}(\sqcup_i \fn_i))^{red} \ . $
\end{theorem}

\begin{proof} Theorem \ref{startsec7} gives us an open cover $\{U_i\}$ with
associated fat
points $\fn_i$ and the following formulas:
\begin{equation}
\nabla_{\fm}U_i \cong U_{i}^{red} \times_{\kappa} \bA_{{\kappa}}^{d(\ell(\fm)-1)} \times_{\kappa}
\nabla_{\fm} \fn_i \ .
\end{equation}
Using Theorem 4.4 of \cite{Sch2}, we may glue the opens $\nabla_{\fm}U_i$ to obtain
the expression in (a). Part (b) follows from (a) because the first two factors on the right hand side of (a) 
are reduced. 
\end{proof}

\begin{remark} \label{conjsec7} Let $X\in\sch{\kappa}$ be such that $X^{red}$ is smooth.
Let $\fn_i$ be as in the previous theorem. If each $\fn_i$ belongs to a simple point system $\bX_i$ where 
$\fx_i = \colim \bX_i$.
Let $\fx$ be the disjoint union of the $\fx_i$. Then $X$ is rationally\footnote{See \S
\ref{8}.} $\fx$-stable. Strictly speaking this doesn't make sense as $\fx$ is colimit of finite $\kappa$-schemes which are not necessarily connected, but the definitions easily generalize to this situation.
\end{remark}

\subsection{The reduced measure.}\label{8}

\begin{definition}
We say that a scheme $X$ is {\it rationally} $\fx$-{\it laxly stable} if it is
weakly stable and the induced map $$(\nabla_{\fx/\fm}X)^{red} \to
(\nabla_{\fx/\fn}X)^{red}$$
is a piecewise trivial fibration with fiber $\bA_{{\kappa}}^{r}$ when $\ell(\fn) >>0$.
If, in addition, $r =d(\ell(\fm)-\ell(\fn))$ with $d=\dim X$, then we say that $X$ is {\it
rationally} $\fx$-{\it stable}.
\end{definition}

As in \S \ref{1}, we may define a function $$rs : \sch{\kappa}\times\Arc_{\kappa} \to
\bN\cup\{+\infty\}$$ by $rs_{\fx}(X) = \ell(\fn)$ where $\fn$ is the minimum fat
point in $\bX$ which satisfies 
the definition for rational $\fx$-stability. When $X$ is rationally
$\fx$-stable, we say that $X$ is {\it rationally} $\fx$-{\it stable at level}
$rs_{\fx}(X)-1$. We denote the collection of all rationally 
$\fx$-stable schemes as $$RS_{\fx} := rs_{\fx}^{-1}(\bN)\ .$$ 
For each $\fx\in\Arc_{\kappa}$, we have a set map 
\begin{equation*}
\lambda_{\fx} : RS_{\fx} \to \hat\sG_{\kappa}
\end{equation*} defined by 
\begin{equation*}
\lambda_{\fx}(X) := [(\nabla_{\fx/\fn} X)^{red}]\bL^{-\ell(\fn)d}
\end{equation*}
 where $d = \dim X$ and $\ell(\fn)\geq rs_{\fx}(X)$. 

\begin{definition}
Let $X$ be an element of $RS_{\fx}$, $B$ a subsieve of $\nabla_{\fx}X$, and let 
$\beta: B\to \bZ\cup\{+\infty\}$ be a function such that 
for all $\fm \in\bX$ with $\ell(\fm)\geq rs_{\fx}(X)$,
the assignment $\pi_{\fm}^{\fx}(B)$  is a formal sieve. In this case, we say that $\alpha$ is a rationally $\fx$-stable function.
 We say that a subsieve $A$ of $\nabla_{\fx}X$ is a {\it rationally} $\fx$-{\it stable}
if its characteristic function is rationally $\fx$-stable. 
\end{definition}

For any rationally $\fx$-stable subsieve $B$ of $\nabla_{\fx}X$, we define
$$\lambda_{\fx}(B) := [(\pi_{\fn}^{\fx}(B))^{red}]\bL^{-d\ell(\fn)}$$ where
$\ell(\fn)\geq rs_{\fx}(X)$ and $d=\dim X$. Let $\bB_{\fx}^{X}[rs]$ denote the collection of
all rational $\fx$-stable subsets of $\nabla_{\fx} X$.

\begin{theorem} \label{measurelambda}
Let $X$ be rationally $\fx$-stable and assume that $d=\dim X$. We
have a set map $$\lambda_{\fx} : \bB_{\fx}^{X}[rs] \to \hat\sG_{\kappa}$$ with the
following properties:
 
\noindent(a) For $B \in \bB_{\fx}^{X}[rs]$, we may define $$
\lambda_{\fx}(B):=[(\pi_{\fn}^{\fx}(B))^{red}]\bL^{-d\ell(\fn)} \in
\hat\sG_{\kappa} \ .$$

\noindent(b) Let $X$ be generically smooth and $S$  a closed sub-scheme of $X$ with $\dim S < \dim X$, then
it is the case that $\lambda_{\fx}(\nabla_{\fx}X_S) = 0 $.

\noindent(c) When $\{B_i\}$ is a countable collection of mutually disjoint
elements of $\bB_{\fx}^{X}[rs]$, then we may define $$\lambda_{\fx}(\cup_i B_i) := \sum_{i}
\lambda_{\fx}(B_i) \ ,$$
whenever the right hand side converges in $\hat\sH_{\kappa}$.

\end{theorem}

\begin{proof} The is exactly the same as the proof of Theorem \ref{measuremu} and Theorem \ref{samesame}. 
\end{proof}

 Without loss of generality, we can assume $Z$ is irreducible and affine in the following construction.
Following \S 2 of \cite{Sch1}, for any $Z \in\sch{\kappa}$ and any constructible subset $F$ of $Z(\kappa)$,
we define the constructible cone $\fC_{Z}(F)$ of $F$ over $Z$ to be the sieve
\begin{equation*}
\fC_{Z}(F)(\fm) := \rho_{\fm}^{-1}(F)
\end{equation*}
 where $\rho_{\fm} : Z(\fm) \to Z({\kappa})$ is the set map induced the residue field morphism $\rho_{\fm}: \spec \kappa \to \fm$. Moreover, if $F$ is a constructible subset of $Z(\kappa)$, then it can be identified (as sets) with a constructible subset of $Z^{red}(\kappa)$ which we will denote by $F^{red}$. Then, it is natural to define 
$\fC_{Z}(F^{red})(\fm) := \fC_{Z}(F)(\fm)$.

\begin{definition} \label{defbarmu}
Let $X$ be a rationally $\fx$-stable scheme and let $B \in
\bB_{\fx}^{X}
[rs]$. 
We define a set map $\bar\mu_{\fx} : \bB_{\fx}^{X}[rs] \to \hat\sH_{\kappa}$ by
$$\bar\mu_{\fx}(B) 
:= [\fC_{\nabla_{\fn}X}((\pi_{\fn}^{\fx}(B)(\kappa))^{red})]\bL^{-d\ell(\fn)} \ ,$$
where $d=\dim X$.
\end{definition}
This is notationally heavy. Thus we will often just write $(\pi_{\fn}^{\fx}(B))^{red}$ in place of $(\pi_{\fn}^{\fx}(B)(\kappa))^{red}$. In other words, $(\pi_{\fn}^{\fx}(B))^{red}$ is the constructible subset of $(\nabla_{\fn}X)^{red}$ formed by the inverse image of the constructible subset of $\nabla_{\fn}X(\kappa)$ determined by the $\kappa$-rational points of the formal sieve $\pi_{\fn}^{\fx}(B)$. See the proof of Theorem \ref{reductionring} for more details.

\begin{proposition} Definition \ref{defbarmu} is independent of choice of $\fn \in \bX$
for large
enough $\ell(\fn)$. 
\end{proposition}
\begin{proof} 
In this proof, we find it useful to somewhat abuse notation and denote the representable functor 
$(\mathbb{A}_{\kappa}^r)^{\circ}=\Mor_{\kappa}( -, \mathbb{A}_{\kappa}^r)$ by $\lef^r(-)$.
Also, we will set $m=\ell(\fm)$ and $n=\ell(\fn)$ for simplicity of notation.
At any rate, we may reduce to the case where there exists an isomorphism of
constructible sets 
\begin{equation*}
j: (\nabla_{\fx/\fm}X)^{red} \iso
(\nabla_{\fx/\fn}X)^{red}\times_{\kappa}\bA_{{\kappa}}^{d(m-n)}
\end{equation*}
This induces a rational isomorphism of sieves
\begin{equation*}
j_{-}^{\circ}: ((\nabla_{\fx/\fm}X)^{red})^{\circ}(-) \iso
((\nabla_{\fx/\fn}X)^{red})^{\circ}(-)\times \bL^{d(m-n)}(-)
\end{equation*}
Now, $a$ is a $\fv$-rational point of
$\fC_{\nabla_{\fm}X}((\pi_{\fm}^{\fx}(B))^{red})
(\fv)$ if and only if $$a(\fv) \in (\pi_{\fm}^{\fx}(B))^{red}\ ,$$ and this is
true if and only if $$j(a(\fv)) 
\in (\pi_{\fn}^{\fx}(B))^{red}\times_{\kappa}\bA_{{\kappa}}^{d(m-n)}\ .$$ Equivalently,
$j_{\fv}^{\circ}
(a)= j\circ a$ is a $\fv$-rational point of
$((\pi_{\fn}^{\fx}(B))^{red})^{\circ}(\fv) \times \bL^{d(m-n)}(\fv)$ such 
that $$(j\circ a)(\fv)=(b\times_{\kappa} c)(\fv) \in
(\pi_{\fn}^{\fx}(B))^{red}\times_{\kappa}\bA_{{\kappa}}^{d(m-n)} \ ,$$
where $b$ is a $\fv$-rational point of
$((\pi_{\fn}^{\fx}(B))^{red})^{\circ}(\fv)$ and $c$ is a $\fv$-rational point of
$\bL^{d(m-n)}(\fv)$. This is true if and only if $j_{\fv}^{\circ}(a)$ is an
$\fv$-rational point of $$\fC_{\nabla_{\fn}X}((\pi_{\fn}^{\fx}(B))^{red})(\fv)
\times \fC_{\bA_{{\kappa}}^{d(m-n)}}(\bA_{{\kappa}}^{d(m-n)})(\fv) \ .$$
By Proposition 7.1 of \cite{Sch1},
$\fC_{\bA_{{\kappa}}^{d(m-n)}}(\bA_{{\kappa}}^{d(m-n)})\cong \bA_{{\kappa}}^{d(m-n)}$. Therefore,
$j_{-}^{\circ}$ induces a rational isomorphism of sieves
$$\fC_{\nabla_{\fm}X}((\pi_{\fm}^{\fx}(B))^{red}) \cong
\fC_{\nabla_{\fn}X}((\pi_{\fn}^{\fx}(B))^{red}) \times \bL^{d(m-n)}(-) \ .  $$
\end{proof}

\begin{theorem} \label{measurebarmu}
Let $X$ be rationally $\fx$-stable and assume that $d=\dim X$. We
have a set map $$\bar\mu_{\fx} : \bB_{\fx}^{X}[rs] \to \hat\sH_{\kappa}$$ with the
following properties:
 
\noindent(a) For $B \in \bB_{\fx}^{X}[rs]$, we have $$\bar\mu_{\fx}(B) :=
[\fC_{\nabla_{\fn}X}((\pi_{\fn}^{\fx}(B))^{red})]\bL^{
-d\ell(\fn)} \in \hat\sH_{\kappa} \ .$$

\noindent(b) If $X$ is generically smooth and when $S$ is a closed sub-scheme of $X$ with $\dim S < \dim X$, then
it is the case that $\bar\mu_{\fx}(\nabla_{\fx}X_S) = 0 $.

\noindent(c) When $\{B_i\}$ is a countable collection of mutually disjoint
elements of $\bB_{\fx}^{X}[rs]$, then we may define $$\bar\mu_{\fx}(\cup_i B_i) := \sum_{i}
\bar\mu_{\fx}(B_i) \ . $$

\end{theorem}
\begin{proof} Again, the proof is the same as  Theorem \ref{measuremu} and Theorem \ref{samesame}. 
\end{proof}

\begin{lemma}\label{triangle}
For any rationally $\fx$-stable scheme $X$, the following diagram commutes:
\large
$$\xymatrix{
&\ &\bB_{\fx}^{X}[rs]\ar[d]^{\lambda_{\fx}} \ar[ld]_{\bar\mu_{\fx}} 
  \\
&\hat\sH_{\kappa} \ar[r]_{\hat\sigma} &\hat\sG_{\kappa}}$$
\normalsize
\end{lemma}
\begin{proof} Let $X$ be a rationally $\fx$-stable scheme. For any $B\in
\bB_{\fx}^{X}[rs]$, we have
$$\hat\sigma(\fC_{\nabla_{\fn}X}((\pi_{\fn}^{\fx}(B))^{red}) ) =
\fC_{\nabla_{\fn}X}((\pi_{\fn}^{\fx}(B))^{red})({\kappa}) = (\pi_{\fn}^{\fx}(B))^{red}
\ , $$ from which the lemma follows.
\end{proof}

\begin{lemma} \label{endsec8}
If $X$ is a $\fx$-stable scheme, then $X$ is also rationally $\fx$-stable and
$\bB_{\fx}^{X}[s] = \bB_{\fx}^{X}[rs]$. Moreover,  assume that $X$ is generically smooth. Then, given any $\fx$-stable
sieve $A\in\Gamma_{\fx}^{X}[s]$ $($i.e., $A$ is constructible $\fx$-stable$)$, we have 
\begin{equation*}
\mu_{\fx}(A) = \bar\mu_{\fx}(A)  \ .
\end{equation*}
\end{lemma}
\begin{proof} 
Let $A\in\Gamma_{\fx}^{X}[s]$. We reduce immediately to the case where $X$ is connected and smooth, which implies that $\nabla_{\fn}X$ is reduced.  For any fat point $\fv$, 
\begin{equation*}
(\pi_{\fn}^{\fx}(A))(\fv) =
\fC_{\nabla_{\fn}X}((\pi_{\fn}^{\fx}(A))({\kappa}))(\fv) \ .
\end{equation*}
Thus,  $[\pi_{\fn}^{\fx}(A)] =
[\fC_{\nabla_{\fn}X}((\pi_{\fn}^{\fx}(A))(\kappa))]$ in $\sH_{\kappa}$ which proves the
lemma. 
\end{proof}

As being rationally $\fx$-stable is a weaker condition than $\fx$-stable, we can
consider $\bar\mu_{\fx}$ as a lift of $\mu_{\fx}$. This is particularly
important in that it allows one to glean information about rationally
$\fx$-stable schemes which we encountered in Theorem \ref{endsec5} and
Conjecture \ref{conjsec7}
without first passing through $\hat\sigma$. 

\subsection{Lax stability and extension of measures}\label{9}

Let $X$ be a rationally $\fx$-laxly stable and let $\fn \in \bX$ be such that
$\ell(\fn) = rls_{\fx}(X)$, we define $\mathbb{S}(\fn) := \{\fm\in\bX \mid \fm
\geq \fn\} $ and a function
\begin{equation*}
l : \mathbb{S}(\fn) \to \bN
\end{equation*}
by $l(\fm) :=l_{\fx}^{X}(\fm) := r-d(\ell(\fm)-\ell(\fn))$ where $r$ is the
unique positive number given to us such that the map $(\nabla_{\fx/\fm}X)^{red} \to
(\nabla_{\fx/\fn}X)^{red}$ is a piecewise trivial fibration with general fiber
$\bA_{{\kappa}}^{r}$ and $d$ is the dimension of $X$. 

 When $X$ is a $\fx$-laxly stable scheme,
we define $\bB_{\fx}^{X}[ls]$ to be the collection of all $\fx$-laxly stable
subsets of $X$. We then have a set map
$\mu_{\fx}^{l} : \bB_{\fx}^{X}[ls] \to \hat\sH_{\kappa}$
defined by 
\begin{equation}\label{eq1}
\mu_{\fx}^{l}(A) := [\pi_{\fm}^{\fx}(A)]\bL^{-d\ell(\fm)-l(\fm)}
\end{equation}
 for large enough $\ell(\fm)$. 

Moreover, when $X$ is a rationally $\fx$-laxly stable scheme,
we define $\bB_{\fx}^{X}[rls]$ to be the collection of all rationally
$\fx$-laxly stable subsets of $X$. We then have a set map
$\bar\mu_{\fx}^{l} : \bB_{\fx}^{X}[rls] \to \hat\sH_{\kappa}$
defined by 
\begin{equation}\label{eq2}
\bar\mu_{\fx}^{l}(A) :=
[\fC_{\nabla_{\fm}X}((\pi_{\fm}^{\fx}(A))^{red})]\bL^{-d\ell(\fm)-l(\fm)}
\end{equation}
 for large enough $\ell(\fm)$.

\begin{remark}
If $X$ is (rationally) $\fx$-stable, then the function $l$ is functionally
equivalent to $0$. Thus, 
\begin{equation*}
\mu_{\fx}^{l} \equiv \mu_{\fx}^{0} \equiv \mu_{\fx} \ ,  \quad \bar\mu_{\fx}^{l}
\equiv \bar\mu_{\fx}^{0} \equiv \bar\mu_{\fx} \ , \quad \lambda_{\fx}^{l} \equiv
\lambda_{\fx}^{0} \equiv \lambda_{\fx} 
\end{equation*}
\end{remark}

\begin{remark}
Note that $\mu_{\fx}^{l}$, $\bar\mu_{\fx}^{l}$, and $\lambda_{\fx}^{l}$ are well
defined -- i.e., they do not depend
on choice of fat point $\fm$ when $\ell(\fm)>>0$. The only work that needs to be
done here is to notice that Theorem \ref{measurebarmu} does not depend on the
dimension of
the general fiber being $d(\ell(m)-\ell(n))$. 
\end{remark}

\begin{theorem} \label{analogue}
The analogue of Theorem \ref{measuremu} (resp., Theorem \ref{measurelambda},
Theorem \ref{measurebarmu}) hold for
$\mu_{\fx}^{l}$, (resp. $\bar\mu_{\fx}^{l}$,  $\lambda_{\fx}^{l}$). The analogue
of
Lemma \ref{endsec8} also holds for $\bar\mu_{\fx}^{l}$ and 
$\lambda_{\fx}^{l}$. Also, the analogue of Theorem \ref{samesame} holds for all of these measures as well.
\end{theorem}

\begin{proof} The proofs are exactly the same as before. \end{proof}

\begin{example} \label{coneex}
Let $(\fx, \bX)$ be in $\Arc_{\kappa}$ and assume that $\bX$ is an eventually simple
point system (Definition \ref{defsimplepoint}). Let $X$ be a scheme of
dimension $d$ which
is smooth over some fat point in $\bX$, and let $\fn$ be the minimum fat point
in $\bX$ such that there exists a smooth morphism $X \to \fn$. Theorem
\ref{endsec5}
assures us that $X$ is rationally $\fx$-laxly stable. Thus, we  compute 
\begin{equation*}
\bar\mu_{\fx}^{l}(\nabla_{\fx}X) =
[\fC_{\nabla_{\fn}X}((\nabla_{\fx/\fn}X)^{red})]\bL^{-d\max(\ell(\fn),\ell(\fm)) -r}
\end{equation*}
where $\fm \in \bX$ is the unique simple fat point such that every $\fv \in \bX$ such that $\fv \geq \fm$ is
simple and where
$r$ is the unique non-negative integer such that $(\nabla_{\fn}\fn)^{red}\cong
\bA_{{\kappa}}^{r}$ (or possibly $(\nabla_{\fm}\fm)^{red}\cong
\bA_{{\kappa}}^{r}$ if we have to apply Theorem \ref{liftingtheorem} when $\fn$ is not simple) 
given to us by the fact that $\bX$ is an eventually simple point system. 
\end{example}

\begin{example} \label{laxex}
Assume that $X$ is an affine scheme of dimension $d$ and that
$X^{red}$ is smooth. Let $\fn$ be the associated fat point of $X$ and assume
that it belongs to a simple point system $\bX$. Theorem \ref{liftingtheorem} and
Theorem
\ref{startsec7}, imply that $X$ is rationally $\fx$-laxly stable at level $0$.
Therefore, we may compute
\begin{equation*}
\bar\mu_{\fx}^{l}(\nabla_{\fx}X) = [\fC_{X}(X^{red})]\bL^{-d} = [X]\bL^{-d} \ .
\end{equation*}
In particular, for any fat point $\fn$ which belongs to a simple point system
$\bX$, we have
\begin{equation*}
\bar\mu_{\fx}^{l}(\nabla_{\fx}\fn) = [\fn]\bL^{-0} = [\fn] \ .
\end{equation*}
We can extend this example by assuming that $\bX$ is merely an eventually simple point system. Then 
\begin{equation*}
 \bar\mu_{\fx}^{l}(\nabla_{\fx}X) = [\fC_{\nabla_{\fn} X}((\nabla_{\fx/\fn}X)^{red})]\bL^{-d\ell(\fn) - r} \ 
\end{equation*}
where  $\fn \in \bX$ is the unique simple fat point such that every $\fv \in \bX$ such that $\fv \geq \fn$ is
simple and where
$r$ is the unique non-negative integer such that $(\nabla_{\fn}\fn)^{red}\cong
\bA_{{\kappa}}^{r}$  
given to us by the fact that $\bX$ is an eventually simple point system. 
Perhaps, more interestingly, when $\bX$ is eventually simple, we have
\begin{equation*}
\bar\mu_{\fx}^{l}(\nabla_{\fx}\fx):= \bar\mu_{\fx}^{l}(\nabla_{\fx}\fv) = [\nabla_{\fn}\fn]\bL^{-r}   \ 
\end{equation*}
where $\fn$ the unique simple fat point such that every $\fv \in \bX$ such that $\fv \geq \fn$ is
simple and $r$ is the unique non-negative integer such that $(\nabla_{\fn}\fn)^{red}\cong
\bA_{{\kappa}}^{r}$. This is because $\fC_{\nabla_{\fn}\fn}((\nabla_{\fn}\fn)^{red}) \cong \nabla_{\fn}\fn$. 
\end{example}

\subsection{Change of variables for $\lambda^l$ and $\bar\mu^l$.}\label{10}

\begin{theorem} \label{changelambda}
Assume that ${\kappa}$ is of characteristic zero.
Let $f : X \to Y$ be a
birational morphism of elements of $\sch\kappa$ of  pure dimension $d$, $Y$ generically smooth,  and let $\alpha : \nabla_{\fx}Y \to
\bZ\cup\{+\infty\}$ be a constructible $\fx$-laxly-stable function, then we have the following
identity in $\hat \sH_{\kappa}$.
$$\int_{\nabla_{\fl}Y} \bL^{-\alpha}d\bar\mu_{\fl}^{l} =
\int_{\nabla_{\fl}X} 
\mathbb{L}^{-\alpha\circ f - J_X[f]}d\bar \mu_{\fl}^{l}$$
when both sides converge.
\end{theorem}

\begin{proof} The proof then is
exactly as in the
proof of Theorem \ref{changemu}.
\end{proof}

\begin{theorem}\label{changebarmu}
Assume that ${\kappa}$ is of characteristic zero.
Let $f : X \to Y$ be a birational
morphism in $\sch\kappa$ with $X$ and $Y$ of pure dimension $d$, $Y$ generically smooth, 
and let $\alpha : \nabla_{\fx}Y \to
\bZ\cup\{+\infty\}$ be a constructible $\fl$-laxly-stable function, then we have the following
identity in $\hat \sG_{\kappa}$: 
$$\int_{\nabla_{\fl}Y} \bL^{-\alpha}d\lambda_{\fl}^{l} =
\int_{\nabla_{\fl}X} 
\mathbb{L}^{-\alpha\circ f - J_X(i,j)[f]}d\lambda_{\fl}^{l} \ .$$
when both sides converge.
\end{theorem}

\begin{proof} Assuming convergence  of the equation in  Theorem \ref{changelambda}, then this is a direct corollary of that theorem and
Theorem \ref{analogue} by
applying $\hat\sigma$ to both sides. However, to get the full statement, one needs to run through the argument in Theorem \ref{changemu}. Note here that we do not have to worry about lifting the sums to $\hat\sH_{\kappa}$.
\end{proof}

As promised earlier, we offer here an extension of $M_{\kappa}$ defined in Theorem \ref{standardgroup}. This is related to the work of Larsen and Lunts found in \cite{LL}.

\begin{theorem} \label{extstandardgroup}
Let $H_{\fx}$ be the subring of $\mathbf{Gr}(\Form_{\kappa})$ generated by finite integral sums of equivalence classes of rationally $\fx$-laxly stable
schemes. Then there is a group homomorphism $M_{\fx} : H_{\fx} \to \sH_{\kappa}$ such that 
\begin{equation*}
 \bar\sigma \circ M_{\fx}|_{S_{\fx}} = M_{{\kappa}} \circ \bar\sigma|_{S_{\fx}} \ 
\end{equation*}
where $S_{\fx}$ is the subring of $H_{\fx}$ generated by finite integral sums of equivalence classes of 
rationally $\fx$-laxly stable schemes with smooth reduction and where $\bar\sigma$ is the homomorphism $\sH_{\kappa}/(\bL) \to \sG_{\kappa}/(\bL)$ 
induced by $\sigma$.
\end{theorem}
\begin{proof}
 As in the proof of Theorem \ref{standardgroup}, it is enough to define $M_{\fx}([X]) = \bar\mu_{\fx}^{l}(\nabla_{\fx}X)$ whenever
 $X$ is a rationally $\fx$-laxly stable scheme and then extend linearly over $\bZ$. The rest follows.
\end{proof}

\begin{remark} \label{extSBremark} As in Remark \ref{SBremark}, we can extend $M_{\fx}$ to a ring homomorphism 
 \begin{equation*}
\bar{M}_{\fx}: H_{\fx}/(\bL)  \to \sH_{\kappa}/(\bL)
\end{equation*}
by defining $\bar{M}_{\fx}([T]) = \bL \cdot \bar{M}_{\fx}([X])$ where $[X]$ modulo $(\bL)$ is equal to $[T]$. We then have
\begin{equation*}
 \bar\sigma \circ \bar{M}_{\fx}|_{S_{\fx}/(\bL)} = \bar{M}_{{\kappa}} \circ \bar\sigma|_{S_{\fx}/(\bL)} \  .
\end{equation*}

\end{remark}

\begin{question} 
 For each $\fx\in \Arc_{\kappa}$, what is $\Ker(\bar{M}_{\fx})$ and does it have a nice geometric interpretation?
\end{question}

Now, $H_{\fx}$ is far from having a simple algebraic description. For example, $[\fv] \in H_{\fl}$ 
where $\fv = \Spec {\kappa}[x,y]/(x^3,y^2,xy)$ because
its coordinate ring is a subring of ${\kappa}[t]/(t^5)$ defined by sending $x$ to $t^2$ and $y$ to $t^3$. Moreover, it is not easy to check that $\fv$ is rationally $\fx$-laxly stable.  Even though $\nabla_{\fl_2}\fv \cong \bA_{\kappa}^2$, I have no idea what $\nabla_{\fl_n}\fv$ looks like explicitly.
 Thus, even the subring of $H_{\fl}$ generated by elements of dimension zero is not easy to describe concretely.

\subsection{Motivic generating series.}\label{11}

Note that in what follows we index $\fn\in\bX$ by the integer $n$. In other words, $\fn$ is the $n$-th truncated jet of some scheme $Y$ at a closed point $o$ --i.e., $\fn = J_{o}^nY$ for some fixed scheme $Y$.
Following \S 9 of \cite{Sch2}, we give the following definition:
\begin{definition} \label{series1}
Let $(\fx, \bX) \in \Arc_{\kappa}$ and  $X\in \sch{\kappa}$ where $\dim X = d$, we define the {\it motivic
igusa-zeta series of} $X$ {\it with respect to} $\fx$ by 
\begin{equation*}
\zeta_{\fx}(X)(t) := \sum_{ n\in \bN} [\nabla_{\fn} X] \bL^{-d\ell(\fn)}t^{n} \
.
\end{equation*}
Moreover, we define {\it the motivic Poincar\'{e} series of} $X$ {\it with respect to} $\fx$
by
\begin{equation*}
P_{\fx}(X)(t) := \sum_{n\in\bN} [\nabla_{\fx/\fn}X] \bL^{-d\ell(\fn)} t^n \ .
\end{equation*}
\end{definition}

\begin{proposition}
When $X$ is $\fx$-stable, then $P_{\fx}(X)$ belongs to $\sH_{\kappa}[t,\frac{1}{1-t}]$.
\end{proposition}
\begin{proof}  By definition of stability, for sufficiently large ${\kappa}$, we have the following equalities
in $\sH_{\kappa}[[t]]$.
\begin{eqnarray} 
P_{\fx}(X)(t) &=& \sum_{n\in\bN} [\nabla_{\fx/\fn}X] \bL^{-d\ell(\fn)} t^n
\nonumber\\
&=&\sum_{n=1}^{{k}} [\nabla_{\fx/\fn}X] \bL^{-d\ell(\fn)} t^n + \sum_{n>{k}}
\mu_{\fx}(\nabla_{\fx}X)t^n \nonumber\\
&=& \sum_{n=1}^{{k}} [\nabla_{\fx/\fn}X] \bL^{-d\ell(\fn)} t^n +
\mu_{\fx}(\nabla_{\fx}X) \frac{t^{k+1}}{1-t}  \ .\nonumber
\end{eqnarray}
\end{proof}

When $X$ is  $\fx$-laxly stable, the same work above will give us the
formula
\begin{equation*}
P_{\fx}(X)(t) = \sum_{n=1}^{{k}} [\nabla_{\fx/\fn}X] \bL^{-d\ell(\fn)} t^n +
\mu_{\fx}^{l}(\nabla_{\fx}X)\sum_{n>k}\gamma_l(n)t^n \ . 
\end{equation*}
Note that $\gamma_l$ was defined  via the function  $l(\fn):=l_{\fx}^{X}(\fn)$. 
When $l$ is linear in $\fn$ meaning that
$l(\fn) = qn+b$, we obtain 
\begin{eqnarray}
P_{\fx}(X)(t) &=& \sum_{n=1}^{{k}} [\nabla_{\fx/\fn}X] \bL^{-d\ell(\fn)} t^n +
\mu_{\fx}^{l}(\nabla_{\fx}X)\bL^{b}\sum_{n>{k}}(\bL^{q}t)^n \nonumber\\
&=& \sum_{n=1}^{{k}} [\nabla_{\fx/\fn}X] \bL^{-d\ell(\fn)} t^n +
\mu_{\fx}^{l}(\nabla_{\fx}X)\bL^{b}\frac{(\bL^{q}t)^{{k}+1}}{1-\bL^{q}t} \ . \nonumber
\end{eqnarray}
Thus, we have the following proposition:
\begin{proposition} \label{poin}
If $X$ is $\fx$-laxly stable and $l$ is a linear function of slope $q$,
then $P_{\fx}(X)$ belongs to $\sH_{\kappa}[t,\frac{1}{1-\bL^{q}t}]$ . 
\end{proposition}
\begin{remark}
It would be easy to generalize this when $l$ is  piece-wise
 linear. Although, I am not sure if the  length function  can ever be piece-wise linear and not linear in this setting.
\end{remark}

\begin{example}\label{rerefinedlaxex}
Let $X$ be an affine scheme of dimension $d$ such that $X^{red}$ is smooth.
Assume that its associated fat point belongs to a simple point system $\bX$ and
let $\fx = \colim\bX$. In Example \ref{laxex}, we showed that $X$ is rationally
$\fx$-laxly stable. Therefore, we can compute the Poincar\'{e} series of $X$
with
respect to $\fx$ when $l$ is a linear function with slope $q$ in the same
way that we proved the previous proposition. In particular, if it is rationally $\fx$-laxly stable at level $0$, then
\begin{equation*}
P_{\fx}(X)(t) =
\bar\mu_{\fx}^{l}(\nabla_{\fx}X)\bL^{b}\frac{t}{1-\bL^{q}t}=[X]\bL^{b-d}\frac{t}
{1-\bL^{q}t} \ .
\end{equation*} 
Therefore, for any fat point $\fn$ which belongs to a simple point system which is rationally $\fx$-laxly stable at level $0$ whose length function has linear slope, we
have
\begin{equation*}
\hat\sigma(P_{\fx}(\fn)(t)) = \bL^{b}\frac{t}{1-\bL^{q}t} \ 
\end{equation*}
where $\hat\sigma$ is extended to a morphism $\hat\sH_{\kappa}((t)) \to \hat\sG_{\kappa}((t))$ by sending $t$ to $t$. 
\end{example}

\begin{remark} If $X$ is $\fx$-laxly stable at level $0$ and $\nabla_{\fx}X \to
\nabla_{\fn}X$ is surjective, then
\begin{equation*} 
 l_{\fx}^{X}(n) = \dim(\nabla_{\fn}X) - d\ell(\fn) \ ,
 \end{equation*}
 which is known in \cite{Sch2} as the defect of $X$ at $\fn$.
 Therefore, when $X$ is rationally $\fx$-laxly stable at $0$, we call
$l_{\fx}^{X}(n)$ the {\it rational} $\fx$ {\it defect} of $X$ at $\fn$. For a simple point
system $\bX$, we have a well defined $l_{\fx}^{\fn}(n)$ rational defect of $\fn$
at $\fn$, which is just equal to $\dim(\nabla_{\fn}\fn),$ which is 
studied in loc. cit. 
 \end{remark}

Following \S 5 of \cite{Sch2}, we define the {\it weightless auto-igusa zeta
series of a limit point} $\fx$ to be
\begin{equation}
\zeta_{\fx}^{w}(t) = \sum_{n\in\bN} \bL^{-l(\fn)}[\nabla_{\fn}\fn]t^n \ . 
\end{equation}
Assume that $(\fx,\bX) \in \Arc_{\kappa}$ where $\bX$ is simple. We have already seen
in our examples that each coefficient in this power series is
$\bar\mu_{\fx}^{l}(\nabla_{\fx}\fn)$ which is in turn just equal to $[\fn]$ since $\bX$ is a simple point system. Thus, in this case, we obtain
\begin{equation*}
\hat\sigma(\zeta_{\fx}^{w}(t)) = \frac{t}{1-t} \ .
\end{equation*}

\bigskip

\bigskip

\noindent\address{Andrew R. Stout\\
 Graduate Center, \\
City University of New York,\\
 365 Fifth Avenue, 10016.}

\end{document}